\documentclass[reqno]{amsart}     
\usepackage{enumerate,amsmath,amssymb}     
\usepackage{amsmath}
\usepackage{pb-diagram}  
\usepackage{pstricks,pst-node,pst-coil,pst-plot}   
\usepackage{graphicx}

\theoremstyle{plain}      
\newtheorem{step}{Step} 
\newtheorem{thm}{Theorem}[section]     
\newtheorem{theorem}[thm]{Theorem}     
     
\newtheorem{corollary}[thm]{Corollary}     
     
\newtheorem{lemma}[thm]{Lemma}     
\newtheorem{prop}[thm]{Proposition}

\newtheorem*{caseI}{Case I}
\newtheorem*{subcaseI.1}{Subcase I.1}
\newtheorem*{subcaseI.2}{Subcase I.2}
\newtheorem*{caseII}{Case II}
\newtheorem*{subcaseII.1}{Subcase II.1}
\newtheorem*{subcaseII.2}{Subcase II.2}
\theoremstyle{remark}

\theoremstyle{definition}

\def\al{{\alpha}}         
         
\def\de{{\delta}}         
         
\def\om{{\omega}}         
\def\Om{{\Omega}}         
\def\la{{\lambda}}
\let\La\Lambda         
\def\ka{{c}}        % Mattias: I changed \kappa to c
\def\si{{\sigma}}         
\def\Si{{\Sigma}}         
\def\ga{{\gamma}}         
\def\ep{{\varepsilon}}         
\def\Ga{{\Gamma}}         
\def\th{{\theta}}         
\let\theta\theta
\def\Th{{\Theta}}         
         
\def\phi{{\varphi}}

\let\pa\partial

\DeclareMathAlphabet{\doba}{U}{msb}{m}{n}

\gdef\mN{\doba{N}}

\gdef\mR{\doba{R}}

\def\grad{{\mathop{\rm grad}}} 
\def\Vol{{\mathop{\rm Vol}}}     
\let\vol\Vol
\def\Scal{{\mathop{\rm Scal}}}     
\def\lamin{\la_{\min}^+} 
\def\lamintilde{\widetilde{\la_{\min}^+}}
\def\taubar{\overline{\tau}}

\def\Cloc{C_{\mathrm{loc}}}

\let\ti\tilde   

\def\OmSpin{{{\Omega}^{\rm spin}_n}}

\def\ddt{\frac{d}{dt}}
\def\d2dt{\frac{d^2}{dt^2}}

\def\eref#1{{\rm (\ref{#1})}}   

\let\<\langle 
\let\>\rangle

\long\def\ignorethis#1{}

\newcommand{\definedas}{\mathrel{\raise.095ex\hbox{:}\mkern-5.2mu=}}

\def\Rmax{R_{\textrm{max}}}

\begin{document}     
%%%%%%%%%%%%%%%%%%%%%%%%%%%%%%%%%%%%%%%%%%%%%%%%%%%%%%%%%%%%%%%%%     

\title{Surgery and the spinorial $\tau$-invariant}     
 
\author{Bernd Ammann} 
\address{NWF I -- Mathematik \\        
Universit\"at Regensburg \\
93040 Regensburg \\  
Germany}
\email{bernd.ammann@mathematik.uni-regensburg.de}

\author{Mattias Dahl} 
\address{Institutionen f\"or Matematik \\
Kungliga Tekniska H\"ogskolan \\
100 44 Stockholm \\
Sweden}
\email{dahl@math.kth.se}

\author{Emmanuel Humbert} 
\address{Institut \'Elie Cartan, BP 239 \\ 
Universit\'e de Nancy 1 \\
54506 Vandoeuvre-l\`es-Nancy Cedex \\ 
France}
\email{humbert@iecn.u-nancy.fr}

\begin{abstract}
We associate to a compact spin manifold $M$ a real-valued invariant
$\tau(M)$ by taking the supremum over all conformal classes of the
infimum inside each conformal class of the first positive Dirac
eigenvalue, when the metrics are normalized to unit volume. This
invariant is a spinorial analogue of Schoen's $\sigma$-constant, also
known as the smooth Yamabe invariant.

We prove that if $N$ is obtained from $M$ by surgery of codimension at
least $2$ then $\tau(N) \geq \min\{\tau(M),\La_n\}$, where $\La_n$ is
a positive constant depending only on $n = \dim M$. Various
topological conclusions can be drawn, in particular that $\tau$ is a
spin-bordism invariant below $\La_n$. Also, below $\La_n$ the values
of $\tau$ cannot accumulate from above when varied over all manifolds
of dimension $n$.
\end{abstract}     

\subjclass[2000]{53C27 (Primary) 55N22, 57R65 (Secondary)}
% Old codim2 {53C27 (Primary) 55N22, 57R65 (Secondary)}
% Old mu2-MSC%%%% 53A30, 35J60(Primary) 35P30, 58J50, 58C40 (Secondary)   

%\date{\today}
\date{August 26, 2008}

\keywords{Dirac operator, eigenvalue, surgery}

\thanks{We thank Victor Nistor for helpful discussions about Sobolev 
inequalities on non-compact manifolds. Also Mattias Dahl thanks the 
Institut \'Elie Cartan, Nancy, for its kind hospitality and support 
during visits when this work was initiated.}

\maketitle     

%Current version \Version, from \Datum, most recent changes by \Person.
%\setcounter{tocdepth}{1}

\tableofcontents

%%%%%%%%%%%%%%%%%%%%%%%%%%%%%%%%%%%%%%%%%%%%%%%%%%%%%%%%%%%%%%     
%%%%%%%%%%%%%%%%%%%%% The text starts here %%%%%%%%%%%%%%%%%%%        
%%%%%%%%%%%%%%%%%%%%%%%%%%%%%%%%%%%%%%%%%%%%%%%%%%%%%%%%%%%%%%

%%%%%%%%%%%%%%%%%%%%%%%%%%%%%%%%%%%%%%%%%%%%%%%%%%%%%%%%%%%%%%
\section{Introduction}
%%%%%%%%%%%%%%%%%%%%%%%%%%%%%%%%%%%%%%%%%%%%%%%%%%%%%%%%%%%%%%

%%%%%%%%%%%%%%%%%%%%%%%%%%%%%%%%%%%%%%%%%%%%%%%%%%%%%%%%%%%%%%
\subsection{Spin manifolds and Dirac operators}
%%%%%%%%%%%%%%%%%%%%%%%%%%%%%%%%%%%%%%%%%%%%%%%%%%%%%%%%%%%%%%

Let $M$ be a compact $n$-dimensional spin manifold without boundary. 
We will always consider spin manifolds as equipped with an orientation 
and a spin structure. The existence of these structures is
equivalent to the vanishing of the first and the second Stiefel-Whitney
classes. 
%In some of the literature, the word ``spin'' only means that
%such structures exist. However, we use the word ``spin'' in the sense
%that $M$ actually comes with a choice of orientation and spin structure.

As explained in \cite{lawson.michelsohn:89, friedrich:00, hijazi:01}
one associates the spinor bundle $\Si_\rho^g M$ to 
the spin structure, together with a Riemannian metric $g$ on $M$ and a
complex irreducible representation $\rho$ of the Clifford algebra over
$\mR^n$. The Dirac operator $D^g_\rho$ is a self-adjoint elliptic
first order differential operator acting on smooth sections of the
spinor bundle $\Si^g M$. It has a spectrum consisting only of real
eigenvalues of finite multiplicity. The spectrum depends on the
choice of spin structure, on the metric~$g$, and a priori on the
representation $\rho$. In even dimensions $n$, the representation
$\rho$ is unique. In odd dimensions there are two choices
$\rho^+$ and $\rho^-$. Exchanging the representation results in
reversing the spectrum, that is if $\lambda$ is an eigenvalue of
$D^g_{\rho^+}$ then $-\lambda$ is an eigenvalue of
$D^g_{\rho^-}$ with the same multiplicity, and vice versa. This has
no effect if $n\equiv 1\mod 4$ since the real/quaternionic structure
on $\Si_\rho^g M $ anti-commutes with the Dirac operator and the
spectrum therefore is symmetric, see \cite[Section 1.7]{friedrich:00}.
However, in dimensions $n\equiv 3 \mod 4$ the choice of $\rho$
matters.  In this case we choose the representation such that Clifford
multiplication of $e_1\cdot e_2 \cdots e_n$ acts as the identity,
where $e_1,\ldots,e_n$ denotes the standard basis of $\mR^n$. We thus
can and will suppress $\rho$ in the notation.

%%%%%%%%%%%%%%%%%%%%%%%%%%%%%%%%%%%%%%%%%%%%%%%%%%%%%%%%%%%%%%
\subsection{The $\tau$-invariant}
%%%%%%%%%%%%%%%%%%%%%%%%%%%%%%%%%%%%%%%%%%%%%%%%%%%%%%%%%%%%%%

We denote by $\la_1^+(D^{\tilde{g}})$ the first non-negative
eigenvalue of $D^{\tilde{g}}$. For a metric $g$ on $M$ we 
define 
$$
\lamin(M,g) 
\definedas 
\inf \la_1^+(D^{\tilde{g}}) \vol(M,\tilde{g})^{1/n},
$$
where the infimum is taken over all metrics $\tilde{g}$ conformal to 
$g$. Further we define
$$
\tau^+(M) \definedas \sup \lamin(M,g),
$$
where the supremum is taken over all metrics $g$ on $M$. This yields
an invariant of the spin manifold $M$. Observe that we do not 
require $M$ to be connected. 

We begin by noting some simple properties of the invariant $\tau^+$.
Let $(S^n,\sigma^n)$ denote the unit sphere with its standard metric.
We have 
$$
\lamin(S^n,\sigma^n) = \frac{n}{2} \omega_n^{1/n},
$$
where $\omega_n$ is the volume of $(S^n,\sigma^n)$. Moreover it is 
shown in \cite{ammann:03,ammann.grosjean.humbert.morel:p07} that 
$$
\lamin(M,g) \leq \lamin(S^n,\sigma^n)
$$
for any compact Riemannian spin manifold $(M,g)$. Together with
Inequality (\ref{ineq.sup.hijazi}) below we get 
$$
\tau^+(S^n) = \lamin(S^n,\sigma^n) =  \frac{n}{2} \omega_n^{1/n},
$$
so for all compact spin manifolds $M$ we have 
$$
\tau^+(M) \leq \tau^+(S^n).
$$ 
If the kernel of $D^g$ is non-trivial, then obviously $\lamin(M,g)=0$.
Conversely, it was shown in \cite{ammann:03} that if the kernel of
$D^g$ is trivial, that is if $D^g$ is invertible, then
$\lamin(M,g)>0$. It follows that $\tau^+(M)>0$ if and only if there is
a metric $g$ on $M$ for which the Dirac operator $D^g$ is
invertible. It is a further fact that $\tau^+(M)=0$ precisely when
$\al(M)\neq 0$, where $\al(M)$ is the alpha-invariant which equals the
index of the Dirac operator for any metric on $M$, see
\cite{ammann.dahl.humbert:p06}.

For compact Riemannian spin manifolds $(M_1,g_1)$ and $(M_2,g_2)$ we
denote by $M_1 \amalg M_2$ the disjoint union of $M_1$ and $M_2$  
with the natural metric $g_1 \amalg g_2$. It is not difficult
to see that    
$$
\lamin(M_1 \amalg M_2, g_1 \amalg g_2) 
= 
\min \{ \lamin(M_1,g_1),\lamin(M_2,g_2) \}.
$$
This implies 
$$
\tau^+(M_1  \amalg M_2) 
= 
\min \{\tau^+(M_1),\tau^+(M_2) \}.
$$
We denote by $-M$ the manifold $M$ equipped with the opposite
orientation. The Dirac operator changes sign when the orientation of
the manifold is reversed. If $M$ has dimension ${}\not\equiv 3\mod 4$
this does not change the first positive eigenvalue of $D$ since the
spectrum is symmetric, so we then have $\lamin(-M,g) = \lamin(M,g)$
and $\tau^+(-M) = \tau^+(M)$. For manifolds $M$ of dimension 
${}\equiv 3\mod 4$ we define $\la_{\rm min}^-(M,g)$ and $\tau^-(M)$
similar to $\lamin(M,g)$ and $\tau^+(M)$ by replacing $\la_1^+$ by the
absolute value of the first non-positive eigenvalue.  We then have
$\lamin(-M,g) = \la_{\rm min}^-(M,g)$ and $\tau^+(-M) = \tau^-(M)$.

%%%%%%%%%%%%%%%%%%%%%%%%%%%%%%%%%%%%%%%%%%%%%%%%%%%%%%%%%%%%%%
\subsection{The $\sigma$-constant}
%%%%%%%%%%%%%%%%%%%%%%%%%%%%%%%%%%%%%%%%%%%%%%%%%%%%%%%%%%%%%%

The $\tau$-invariant is a spinorial analogue of the $\sigma$-constant
\cite{kobayashi:87,schoen:89} which is defined for a compact manifold
$M$ by
\begin{equation*} 
\si(M)
\definedas \sup \inf 
\frac{\int\Scal^{\ti g} \,dv^{\ti g}}
{\Vol(M,{\ti g})^{\frac{n-2}{n}}} ,
\end{equation*}
where the infimum runs over all metrics $\ti g$ in a conformal class
and the supremum runs over all conformal classes. $\si(M)$ is also
known as the smooth Yamabe invariant of $M$. When $\si(M)$ is positive
it can be computed in a way analogous to $\tau^+(M)$ using the
lowest eigenvalue of the conformal Laplacian 
$L^g = 4 \frac{n-1}{n-2} \Delta^g + \Scal^g$ instead of
$\la^+_1(D^g)$. Hijazi's inequality \cite{hijazi:86,hijazi:91} gives a
comparison of the two invariants,
\begin{equation} \label{ineq.sup.hijazi}
\tau^\pm(M)^2 \geq \frac{n}{4(n-1)} \sigma(M).
\end{equation}
For $M=S^n$ equality is attained in \eref{ineq.sup.hijazi}.
Upper bounds for $\tau^\pm(M)$ may help to determine the 
$\sigma$-constant.

Surgery formulas for the $\sigma$-constant analogous to those obtained 
in this paper have been proved in~\cite{ammann.dahl.humbert:p08a}.

%%%%%%%%%%%%%%%%%%%%%%%%%%%%%%%%%%%%%%%%%%%%%%%%%%%%%%%%%%%%%%
\subsection{Geometric constants}
%%%%%%%%%%%%%%%%%%%%%%%%%%%%%%%%%%%%%%%%%%%%%%%%%%%%%%%%%%%%%%

We are going to prove a surgery formula for the invariant $\tau^+$. 
This formula involves geometric constants $\La_{n,k}$ which 
we now define.

For a complete spin manifold $(V,g)$ we set
$$ 
\lamintilde(V,g) 
\definedas
\inf \la\in [0,\infty] ,
$$
where the infimum is taken over all $\la\in (0,\infty)$ for which
there is a non-zero spinor field 
$\phi\in L^\infty(V) \cap L^2(V) \cap \Cloc^1(V)$ such that 
$\|\phi\|_{L^\frac{2n}{n-1}(V)} \leq 1$, and
\begin{equation}  \label{eq.conf} 
D^g \phi
=
\la |\phi|^\frac{2}{n-1} \phi.
\end{equation} 
If there are no such solutions of (\ref{eq.conf}) on $V$ then
$\lamintilde(V,g)=\infty$. 

For a positive integer $k$ we let $\xi^k$ denote the Euclidean metric 
on $\mR^k$. For $\ka \in \mR$ we denote by
$\eta^{k+1}_\ka \definedas e^{2 \ka t} \xi^k + dt^2$ the hyperbolic
metric of sectional curvature $-\ka^2$ on $\mR^{k+1}$. As above 
$\sigma^{n-k-1}$ denotes the metric of sectional curvature $1$ on 
$S^{n-k-1}$. We define the product metric 
$$
G_\ka 
\definedas 
\eta^{k+1}_\ka + \sigma^{n-k-1}
$$
on $\mR^{k+1} \times S^{n-k-1}$, and we define our geometric
constants as  
$$
\Lambda_{n,k} 
\definedas
\inf_{\ka \in [-1,1]} \lamintilde(\mR^{k+1} \times S^{n-k-1}, G_\ka ),
$$
and
$$
\La_{n}\definedas\min_{0\leq k\leq n-2} \La_{n,k}. 
$$
Note that the infimum could as well be taken over $\ka \in [0,1]$
since $G_{\ka}$ and $G_{-\ka}$ are isometric. It is
easy to see that $\La_{n,0} = \lamin(S^n,\si^n)$.  For $k>0$ we are
not able to compute these constants, but at least we can show that
they are positive.
\begin{theorem} \label{lamin_hyperb}
For $0 \leq k \leq n-2$ we have $\Lambda_{n,k}>0$.
\end{theorem}

%%%%%%%%%%%%%%%%%%%%%%%%%%%%%%%%%%%%%%%%%%%%%%%%%%%%%%%%%%%%%%
\subsection{Joining manifolds}
%%%%%%%%%%%%%%%%%%%%%%%%%%%%%%%%%%%%%%%%%%%%%%%%%%%%%%%%%%%%%%

We are going to study the behaviour of $\tau^+$ when two compact
Riemannian spin manifolds are joined along a common submanifold.  
Let~$M_1$ and $M_2$ be spin manifolds of dimension $n$ and let $N$ be
obtained by joining $M_1$ and $M_2$ along a common submanifold as
described in Section \ref{joining}. The manifold $N$ is spin and from
the construction there is a natural choice of spin structure on $N$.
The following results make it possible to compare 
$\tau^+(M_1 \amalg M_2)$ and $\tau^+(N)$.
\begin{theorem} \label{main}
Let $(M_1,g_1)$ and $(M_2,g_2)$ be compact Riemannian spin manifolds
of dimension $n$ for which both $D^{g_1}$ and $D^{g_2}$ have trivial
kernel. Let $W$ be a compact spin manifold of dimension $k$
embedded into $M_1$ and $M_2$ with trivializations of the
corresponding normal bundles given. Assume that $0 \leq k \leq n-2$,
and let $N$ be obtained by joining $M_1$ and $M_2$ along $W$. 
 Then there is a family of metrics $g_\th$, 
$\th\in(0,\th_0)$ on $N$
satisfying
\begin{eqnarray*}
\min \{\lamin(M_1 \amalg M_2, g_1\amalg g_2),\La_{n,k} \} 
&\leq &
\liminf_{\th\to 0}  \lamin(N,g_\th) )\\
 &\leq &\limsup_{\th\to 0}  \lamin(N,g_\th)\\
& \leq&  
\lamin(M_1 \amalg M_2, g_1\amalg g_2 ).
\end{eqnarray*}
\end{theorem}

\noindent Taking the supremum over all metrics on $M_1 \amalg M_2$ the first
inequality gives us the following corollary. 
\begin{corollary} \label{firstcortomain}
In the situation of Theorem \ref{main} we have 
$$
\tau^+(N)
\geq 
\min \{ \tau^+(M_1 \amalg M_2),\Lambda_{n,k} \}
\geq 
\min \{ \tau^+(M_1), \tau^+(M_2),\Lambda_{n} \}.
$$
\end{corollary}
Note that these estimates on $\tau^+$ would be trivial without 
Theorem \ref{lamin_hyperb}.

%%%%%%%%%%%%%%%%%%%%%%%%%%%%%%%%%%%%%%%%%%%%%%%%%%%%%%%%%%%%%%
\subsection{Surgery and bordism}
%%%%%%%%%%%%%%%%%%%%%%%%%%%%%%%%%%%%%%%%%%%%%%%%%%%%%%%%%%%%%%

Performing surgery on a spin manifold is a special case of joining
manifolds, this is discussed in more detail in 
Section~\ref{joining}. From Corollary~\ref{firstcortomain} we
get an inequality relating the $\tau$-invariant before and after
surgery. For a compact spin manifold $M$ of dimension $n$ we define 
$$
\taubar^+(M) \definedas \min \{ \tau^+(M), \La_{n} \}.
$$
We also define
$$
\taubar(M) \definedas \min \{ \tau^+(M),\tau^-(M), \La_{n} \}.
$$
If $n\not\equiv 3 \mod 4$ then $\taubar(M)=\taubar^+(M)$. As noted
before, all results for $\taubar^+(M)$ also hold for 
$\taubar^-(M) \definedas \min\{\tau^-(M), \La_{n} \}$. 

\begin{corollary} \label{cortausurgery}
Assume that $M$ is a spin manifold of dimension $n$
and that $N$ is obtained from $M$ by a surgery of codimension 
$n-k \geq 2$. Then 
$$
\tau^+(N)
\geq 
\min \{ \tau^+(M),\Lambda_{n,k} \}
\geq 
\min \{ \tau^+(M),\Lambda_{n} \}.
$$
\end{corollary}
Corollary~\ref{cortausurgery} tells us that
$$
\taubar^+(N) \geq \taubar^+(M), \quad \taubar(N)\geq \taubar(M).
$$
Two compact spin manifolds $M$ and $N$ are spin bordant if there is a 
spin diffeomorphism from their disjoint union to the boundary of a
spin manifold of one dimension higher, and this diffeomorphism
respects the orientation of $N$ and reverses that of $M$. This happens
if and only if $N$ can be obtained from $M$ by a sequence of
surgeries. To apply Corollary~\ref{cortausurgery} we need to know when
this sequence of surgeries can be chosen to include only surgeries of
codimension at least two. The theory of handle decompositions of
bordisms tells us that this can be done when $N$ is connected, see 
\cite[VII Theorem 3]{kirby:89} for dimension $3$, and 
\cite[VIII Proposition 3.1]{kosinski:93} for higher dimensions.

\begin{corollary} \label{lem.bord.inv}
Let $M$ and $N$ be spin bordant manifolds of dimension at least $3$
and assume that $N$ is connected. Then $\taubar(N)\geq\taubar(M)$.
In particular, if $M$ is also connected we have $\taubar(N)=\taubar(M)$.
\end{corollary}

Corollary~\ref{lem.bord.inv} can also be shown in dimension $2$ with
similar arguments \cite[Theorem 1.3]{ammann.humbert:p06a}. 
 
The spin bordism group $\OmSpin$ is the set of equivalence classes of
spin bordant manifolds of dimension $n$ with disjoint union as
addition. Since every element in $\OmSpin$ can be represented by a
connected manifold we obtain a well-defined map 
$\taubar: \OmSpin \to [0,\La_{n}]$ which sends the equivalence class
$[M]$ of a connected spin manifold $M$ to $\taubar(M)$.

\begin{corollary}
There is a positive constant $\ep_n$ such that
$$
\tau^+(M) \in \{0\} \cup [\ep_n,\lamin(S^n,\sigma^n)].
$$
for all spin manifolds $M$ of dimension $n$.
\end{corollary}

\begin{proof}
The spin bordism group $\OmSpin$ is finitely generated 
\cite[page 336]{stong:68}. This implies that the kernel of the map 
$\al:\OmSpin\to KO_n$ is also finitely generated. Let 
$[N_1], \ldots, [N_r]$ be generators of this kernel, we assume that
the manifolds $N_i$ are all connected. Since $\tau(M) = 0$ if and only
if $\al(M) \neq 0$ we obtain the corollary for 
$$
\ep_n \definedas \min \{ \La_{n}, \taubar(N_1), \dots, \taubar(N_r) \}.
$$ 
\end{proof}

The $\al$-map is injective when $n<8$, and then $\ep_n=\La_n$. 
We do not know whether there are $n \in \mN$ with 
$\ep_n < \La_n$. In other words, we do not know if there are 
$n$-dimensional manifolds $M$ with $0<\tau^+(M)<\La_n$. If such 
manifolds exist, the following observations might be interesting.

First, if $M$ is a spin manifold with $\tau^+(M)<\La_n$, then it 
follows from Corollary~\ref{lem.bord.inv} that the $\si$-constant 
of any manifold $N$ spin bordant to $M$ satisfies
$$
\si(N)\leq \frac{4(n-1)}{n}\tau^+(M)^2.
$$

For the next observation we define
$$
S(t) \definedas \{[M]\in \OmSpin \,|\, \taubar(M)\geq t\}, 
\quad 
S^+(t) \definedas \{[M]\in \OmSpin \,|\, \taubar^+(M)\geq t\},
$$
and 
$$
T(t) \definedas \{[M]\in \OmSpin \,|\, \taubar(M)> t\},
\quad 
T^+(t) \definedas \{[M]\in \OmSpin \,|\, \taubar^+(M)> t\}.
$$
Obviously $S(t)=S^+(t)$ and $T(t)=T^+(t)$ in dimensions 
$n\not\equiv 3 \mod 4$.

\begin{corollary}
$S(t)$ is a subgroup of $\OmSpin$ for $t\in[0,\La_{n}]$ and 
$T(t)$ is a subgroup of $\OmSpin$ for $t\in[0,\La_{n})$. 
If $n\equiv 3 \mod 4$, then $S^+(t)$ and $T^+(t)$ are submonoids.
\end{corollary}

\begin{corollary}
The values of $\taubar$ cannot accumulate from above. 
\end{corollary}

\begin{proof}
Assume that $t_i \definedas \taubar(M_i)$, $i\in \mN$, is a decreasing
sequence of values of $\taubar$ which converges to a limit $t_\infty$.
We want to show that $t_i = t_\infty$ for all but finitely many $i$.
 
We have $S(t_i) \subset S(t_{i+1})$, and hence $\bigcup_i S(t_i) =
T(t_\infty)$ is a subgroup of the finitely generated group
$\OmSpin$. It is thus finitely generated itself and we choose a finite
set of generators. There must then be an $I \in \mN$ such that
$S(t_I)$ contains this finite set, and thus $S(t_I) = T(t_\infty)$. 
Hence $[M_i] \in S(t_I)$ for all $i$, which implies $t_i \geq t_I$.
We conclude that $t_i = t_I = t_\infty$ for $i \geq I$. 
%
%Assume that $t_i \definedas \taubar(M_i)$, $i\in \mN$, is a sequence
%of values of $\taubar$. We want to show that the infimum $\inf t_i$ is 
%attained by a $t_i$.  
%Assume that the infimum is not attained. After passing to
%a subsequence we can assume that the sequence $t_i$ is
%decreasing. Then $S(t_i)\subset S(t_{i+1})$, hence 
%$\bigcup S(t_i)=T(t_\infty)$ is a subgroup of the finitely generated
%group $\OmSpin$. It is thus finitely generated itself. Hence there
%exists $r\in \mN$ such that $S(t_r)$ contains the finite set of
%generators, and thus $S(t_r)=T(t_\infty)$. Hence $[M_i]\in S(t_r)$ for
%all $i$, which implies $t_i\geq t_r$, i.e. we obtain the contradiction
%$t_r = \inf t_i$.
\end{proof}

We do not know whether $\taubar^+$ can accumulate from above in
dimensions $n\equiv 3\mod 4$.

%%%%%%%%%%%%%%%%%%%%%%%%%%%%%%%%%%%%%%%%%%%%%%%%%%%%%%%%%%%%%%%%%
\subsection{Variants of the results}
%%%%%%%%%%%%%%%%%%%%%%%%%%%%%%%%%%%%%%%%%%%%%%%%%%%%%%%%%%%%%%%%%

We already remarked earlier that if the alpha-genus $\al(M)$ of a spin
manifold $M$ does not vanish, then the index theorem tells us that the
kernel of $D^g$ is non-trivial for any metric $g$ on $M$, and hence
$\tau^+(M)=0$. For a connected spin manifold $M$ the index theorem
implies that the kernel of the Dirac operator has at least dimension 
$$
a(M)
\definedas
\begin{cases}
|\widehat{A}(M)|, &\text{if $n \equiv 0 \mod 4$;} \\
1, &\text{if $n \equiv 1 \mod 8$ and $\alpha(M)\neq 0$;} \\
2, &\text{if $n \equiv 2 \mod 8$ and $\alpha(M)\neq 0$;} \\
0, &\text{otherwise.}\\  
\end{cases}
$$

Let us modify the definition of $\tau^+$ and use the $k$-th 
non-negative eigenvalue of the Dirac operator instead of the first
one. The quantity thus obtained, denoted by $\tau^+_k(M)$, is zero if
$k\leq a(M)$. It follows from \cite{ammann:03} and
\cite{ammann.dahl.humbert:p06} that $\tau^+_{a(M)+1}(M)>0$. We expect
that our methods generalize to this situation and yield similar
surgery formulas for $\tau^+_k$.

%%%%%%%%%%%%%%%%%%%%%%%%%%%%%%%%%%%%%%%%%%%%%%%%%%%%%%%%%%%%%%%%% 
\section{Preliminaries}
\label{sec.prelim} 
%%%%%%%%%%%%%%%%%%%%%%%%%%%%%%%%%%%%%%%%%%%%%%%%%%%%%%%%%%%%%%%%% 

%%%%%%%%%%%%%%%%%%%%%%%%%%%%%%%%%%%%%%%%%%%%%%%%%%%%%%%%%%%%%%%%% 
\subsection{Notation for balls and neighbourhoods} 
\label{notation}
%%%%%%%%%%%%%%%%%%%%%%%%%%%%%%%%%%%%%%%%%%%%%%%%%%%%%%%%%%%%%%%%% 

We write $B^n(r)$ for the open ball of radius $r$ around $0$
in $\mR^n$, and set $B^n \definedas B^n(1)$. For a Riemannian manifold 
$(M,g)$ we let $B^g(p,r)$ denote the open ball of radius $r$ around 
$p \in M$. If the Riemannian metric is clear from the context we will
write $B(p,r)$. For a Riemannian manifold $(M,g)$ and a subset 
$S \subset M$ we let $U^g(S,r) \definedas \bigcup_{x \in S} B^g(x,r)$
denote the $r$-neighbourhood of $S$. Again, if the Riemannian metric
is clear from the context we abbreviate to $U(S,r)$.

%%%%%%%%%%%%%%%%%%%%%%%%%%%%%%%%%%%%%%%%%%%%%%%%%%%%%%%%%%%%%%%%% 
\subsection{Joining manifolds along submanifolds} 
\label{joining}
%%%%%%%%%%%%%%%%%%%%%%%%%%%%%%%%%%%%%%%%%%%%%%%%%%%%%%%%%%%%%%%%% 

We are now going to describe how two manifolds are joined
along a common submanifold with trivialized normal bundle. Strictly
speaking this is a differential topological construction, but since we
work with Riemannian manifolds we will make the construction adapted
to the Riemannian metrics and use distance neighbourhoods defined by
the metrics etc. 

Let $(M_1,g_1)$ and $(M_2,g_2)$ be complete Riemannian manifolds of
dimension $n$. Let $W$ be a compact manifold of dimension $k$, where 
$0 \leq k \leq n$. We assume that $W$ is embedded in both $M_1$ and
$M_2$ with trivializations of the normal bundle, we desribe these
embeddings as follows.

Let $\bar{w}_i: W \times \mR^{n-k} \to TM_i$,
$i=1,2$, be smooth embeddings. We assume that $\bar{w}_i$ restricted
to $W \times \{ 0 \}$ maps to the zero section of $TM_i$ (which we
identify with $M_i$) and thus gives an embedding $W \to M_i$. The
image of this embedding is denoted by $W_i'$. Further we assume that
$\bar{w}_i$ restrict to linear isomorphisms $\{ p \} \times \mR^{n-k}
\to N_{\bar{w}_i(p,0)} W_i'$ for all $p \in W_i$, where $N W_i'$
denotes the normal bundle of $W_i'$ defined using $g_i$.

We set $w_i \definedas \exp^{g_i} \circ \bar{w}_i$. For $i=1,2$ this
gives embeddings $w_i: W \times B^{n-k}(\Rmax) \to M_i$ for some 
$\Rmax > 0$. We have $W_i' = w_i(W \times \{ 0 \})$ and we define
the disjoint union
$$ 
(M,g) \definedas (M_1 \amalg M_2, g_1 \amalg g_2),
$$
and 
$$
W' \definedas W_1' \amalg W_2'.
$$
Let $r_i$ be the function on $M_i$ giving the distance to $W_i'$. 
Then $r_1 \circ w_1 (w,x)= r_2 \circ w_2(w,x)=|x|$ for $w\in W$, 
$x\in B^{n-k}(\Rmax)$. Let
$r$ be the function on $M$ defined by $r(x) \definedas r_i(x)$ for 
$x \in M_i$, $i=1,2$. For $0 < \ep$ we set 
$U_i(\ep) \definedas \{ x \in M_i \, : \, r_i(x) < \ep \}$ and
$U(\ep) \definedas U_1(\ep) \cup U_2(\ep)$. For $0 < \ep < \th$ we
define
$$
N_{\ep} 
\definedas
( M_1 \setminus U_1(\ep) ) \cup ( M_2 \setminus U_2(\ep) )/ {\sim},
$$
and 
$$
U^N_\ep (\th)
\definedas
(U(\th) \setminus U(\ep)) / {\sim}
$$
where ${\sim}$ indicates that we identify $x \in \partial U_1(\ep)$
with $w_2 \circ w_1^{-1} (x) \in \partial U_2(\ep)$. Hence
$$
N_{\ep} 
=
(M \setminus U(\th) ) \cup U^N_\ep (\th).
$$
We say that $N_\ep$ is obtained from $M_1$, $M_2$ (and $\bar{w}_1$, 
$\bar{w}_2$) by a connected sum along $W$ with parameter $\ep$. 

The diffeomorphism type of $N_\ep$ is independent of $\ep$, hence we
will usually write $N = N_\ep$. However, in some situations where
dropping the index $\ep$ might cause ambiguites we will write
$N_\ep$. For example the function $r:M_1 \amalg M_2\to [0,\infty)$ 
also defines a continuous function $r:N_\ep\to [\ep,\infty)$ whose 
definition depends on $\ep$. We will also keep the $\ep$-subscript for  
$U_\ep^N(\th)$ as important estimates for spinors will be carried 
out on  $U_\ep^N(\th)$.
%, and these estimates are not invariant if one
%applies a diffeomorphism of $M$ to  $U_\ep^N(\th)$ without applying it
%to the spinor. 
As the embeddings $w_1$ and $w_2$ preserve the spin
structure, the manifold $N$ carries a spin structure such that its
restriction to $(M_1\setminus w_1(W\times B^{n-k})) \amalg  
(M_2\setminus w_2(W\times B^{n-k}) )$ coincides with the restriction
of the given spin structure on $M_1 \amalg M_2$. If $W$ is not
connected, then this choice is not unique. The statements of our 
theorem hold for any such spin structure on $N$. 

The surgery operation on a manifold is a special case of taking
connected sum along a submanifold. Indeed, let $M$ be a compact
manifold of dimension $n$ and let $M_1 = M$, $M_2 = S^n$, 
$W = S^k$. Let $w_1 : S^k \times B^{n-k} \to M$ be an embedding
defining a surgery and let $w_2 :  S^k \times B^{n-k} \to S^n$ be the
standard embedding. Since $S^n \setminus w_2 (S^k \times B^{n-k})$ is
diffeomorphic to $B^{k+1} \times S^{n-k-1}$ we have in this situation
that $N$ is obtained from $M$ using surgery on $w_1$, see 
\cite[Section VI.9]{kosinski:93}.

%%%%%%%%%%%%%%%%%%%%%%%%%%%%%%%%%%%%%%%%%%%%%%%%%%%%%%%%%%%%%%%%%
\subsection{Comparing spinors for different metrics}  
\label{bgt}
%%%%%%%%%%%%%%%%%%%%%%%%%%%%%%%%%%%%%%%%%%%%%%%%%%%%%%%%%%%%%%%%%

Let $M$ be a spin manifold of dimension~$n$ and let $g$, $g'$ be
Riemannian metrics on $M$. The goal of this paragraph is to identify
the spinor bundles of $(M,g)$ and $(M,g')$ following Bourguignon and
Gauduchon \cite{bourguignon.gauduchon:92}.   

There exists a unique endomorphism $b^{g}_{g'}$ of $TM$ which is
positive, symmetric with respect to $g$, and satisfies 
$g(X,Y) = g'(b^{g}_{g'} X, b^{g}_{g'}Y)$ for all $X,Y \in TM$.
This endomorphism maps $g$-orthonormal frames at a point to
$g'$-orthonormal frames at the same point and we get a map
$b^{g}_{g'}: \mathrm{SO}(M,g) \to \mathrm{SO}(M,g')$ of
$\mathrm{SO}(n)$-principal bundles. If we assume that
$\mathrm{Spin}(M,g)$ and $\mathrm{Spin}(M,g')$ are equivalent spin
structures on $M$ then the map $b^{g}_{g'}$ lifts to a map
$\beta^{g}_{g'}$ of $\mathrm{Spin}(n)$-principal bundles,
\dgARROWLENGTH=2em
$$
\begin{diagram}   
\node{\mathrm{Spin}(M,g)}\arrow[2]{e,t}
{\beta^{g}_{g'}}\arrow{s}\node[2]
{\mathrm{Spin}(M,g')}\arrow{s}\\   
\node{\mathrm{SO}(M,g)}\arrow[2]{e,t}
{b^{g}_{g'}}\node[2]{\mathrm{SO}(M,g')}   
\end{diagram}.
$$
From this we get a map between the spinor bundles $\Sigma^{g} M$ and
$\Sigma^{g'} M$ denoted by the same symbol and defined by    
\begin{equation*} 
\begin{aligned}
\beta^{g}_{g'}:
\Sigma^{g} M = 
\mathrm{Spin}(M,g) \times_\rho \Sigma_n 
&\to
\mathrm{Spin}(M,g')\times_\rho \Sigma_n 
= \Sigma^{g'} M , \\   
\psi=[s,\phi]
&\mapsto
[\beta^{g}_{g'} s,\phi]  = \beta^{g}_{g'} \psi ,
\end{aligned}
\end{equation*} 
where $(\rho,\Sigma_n)$ is the complex spinor representation, and 
where $[s,\phi] \in \mathrm{Spin}(M,g) \times_\rho \Sigma_n$ denotes
the equivalence class of 
$(s,\phi) \in \mathrm{Spin}(M,g) \times \Sigma_n$ for the equivalence
relation given by the action of $\mathrm{Spin}(n)$. The map
$\beta_{g'}^g$ of Hermitian vector bundles is fiberwise an isometry.

We define the Dirac operator ${}^{g\mkern-4mu}D^{g'}$ acting on
sections of the spinor bundle for $g$ by
$$
{}^{g\mkern-4mu}D^{g'}
\definedas
(\beta^{g}_{g'})^{-1} \circ D^{g'} \circ \beta^{g}_{g'} .
$$
In \cite[Theorem 20]{bourguignon.gauduchon:92} the operator 
${}^{g\mkern-4mu}D^{g'}$ is computed in terms of $D^g$ and some extra
terms which are small if $g$ and $g'$ are close. Formulated in a way
convenient for us the relationship is 
\begin{equation} \label{relD}
{}^{g\mkern-4mu}D^{g'} \psi 
= 
D^g \psi 
+
A^{g}_{g'}(\nabla^g \psi) 
+
B^{g}_{g'}(\psi) ,
\end{equation}
where $A^{g}_{g'} \in \hom(T^*M \otimes \Sigma^{g} M, \Sigma^{g} M)$
satisfies 
\begin{equation} \label{boundA^g_g'}
| A^{g}_{g'} | \leq C | g - g' |_g , 
\end{equation}
and $B^{g}_{g'} \in \hom(\Sigma^{g} M, \Sigma^{g} M)$ satisfies 
\begin{equation} \label{boundB^g_g'}
| B^{g}_{g'} | \leq C ( | g - g' |_g + | \nabla^g (g - g') |_g ) 
\end{equation}
for some constant $C$.

In the special case that $g'$ and $g$ are conformal with $g'= F^2 g$
for a positive smooth function $F$ the formula simplifies considerably,
and one obtains  
\begin{equation} \label{confD}
{}^{g\mkern-4mu}D^{g'}( F^{-\frac{n-1}{2}} \psi) 
= 
F^{-\frac{n+1}{2}} D^g \psi,
\end{equation}
see for instance \cite{hitchin:74, baum:81}.

%%%%%%%%%%%%%%%%%%%%%%%%%%%%%%%%%%%%%%%%%%%%%%%%%%%%%%%%%%%%%%%%%
\subsection{Regularity results}  
\label{regularity}
%%%%%%%%%%%%%%%%%%%%%%%%%%%%%%%%%%%%%%%%%%%%%%%%%%%%%%%%%%%%%%%%%

By standard elliptic theory we have the following lemma (see for
example \cite[Chapter 3]{ammann:habil} where the corresponding results
of \cite{gilbarg.trudinger:77} are adapted to the Dirac operator). 

\begin{lemma}
Let $(V,g)$ be a Riemannian spin manifold and $\Omega \subset V$ an
open set with compact closure in $V$. Let also $r \in  (1,\infty)$. 
Then there is a constant $C$ so that 
\begin{equation} \label{ellip_est}
\int_{\Omega} |\nabla^g \phi|^r\,dv^g 
\leq 
C \left( \int_{\Omega} |D^g \phi|^r\,dv^g 
+ \int_{\Omega} | \phi|^r\,dv^g \right)
\end{equation}
for all $\phi \in \Gamma(\Sigma^{g} {\Omega})$ which are 
of class $C^1$ and compactly supported in $\Om$. 
\end{lemma}

For a compact Riemannian manifold with invertible Dirac operator we
have the following special case.

\begin{lemma}
Let $(V,g)$ be a compact Riemannian spin manifold such that 
$D^g$ is invertible. Then there exists a constant $C$ such that 
\begin{equation} \label{ellip_est2}
\int_V |\nabla^g \phi|^{\frac{2n}{n+1}} \,dv^g 
\leq 
C \int_V |D^g \phi|^{\frac{2n}{n+1}} \,dv^g
\end{equation}
for all $\phi \in \Gamma(\Sigma^g V)$ of class $C^1$. 
\end{lemma}

%%%%%%%%%%%%%%%%%%%%%%%%%%%%%%%%%%%%%%%%%%%%%%%%%%%%%%%%%%%%%%%%%
\subsection{The associated variational problem} 
\label{functional}
%%%%%%%%%%%%%%%%%%%%%%%%%%%%%%%%%%%%%%%%%%%%%%%%%%%%%%%%%%%%%%%%%

Let $(M,g)$ be a compact spin manifold of dimension~$n$ with 
$\ker D^g=\{0\}$. We define the functional $J^g$ acting on smooth
spinor fields $\psi \in \Gamma(\Sigma^g M)$ by
$$
J^g(\psi)
\definedas
\frac{\left( 
\int_M |D\psi|^{\frac{2n}{n+1}} \,dv^g 
\right)^\frac{n+1}{n}}
{\int_M \< D\psi,\psi\>\,dv^g},
$$
whenever the denominator is non-zero. Using techniques from
\cite{lott:86} it was proved in \cite{ammann:03} that    
\begin{equation} \label{funct}   
\lamin(M,g)
=
\inf_\psi J^g(\psi) ,   
\end{equation}   
where the infimum is taken over the set of smooth spinor fields
satisfying  
$$
\int_M  \< D\psi,\psi\>\,dv^g  > 0.
$$   
If $g$ and $\tilde{g}= F^2 g$ are conformal metrics on $M$ and if
$J^g$ and $J^{\tilde{g}}$ are the associated functionals, then by
Relation (\ref{confD}) one computes that 
\begin{equation} \label{funct_conf}
J^{\tilde{g}} (F^{-\frac{n-1}{2}}\psi) = J^g(\psi)
\end{equation} 
for smooth $\psi \in \Ga (\Si^g M))$.

The following result gives a universal upper bound on $\lamin(M,g)$.
\begin{prop} \label{aubin} 
Let $(M,g)$ be a compact spin manifolds of dimension $n \geq 2$.      
Then
\begin{equation} \label{aubin1}
\lamin(M,g) \leq \lamin(S^n,\sigma^n) 
= \frac{n}{2} \, \om_n^{1/n}, 
\end{equation} 
where $\om_n$ is the volume of $(S^n,\sigma^n)$.
\end{prop}
Proposition \ref{aubin} was proven for $n\geq 3$ in \cite{ammann:03}
using geometric methods. In the case $n=2$ the article
\cite{ammann:03} only provides a proof if $\ker D=\{0\}$. Another
method that yields the proposition in full generality is to construct
for any $p\in M$ and $\ep>0$ a suitable test spinor field $\psi_\ep$
supported in $B^g(p,\ep)$ satisfying 
$J^g (\psi_\ep)\leq \lamin(S^n,\si^n)+o(\ep)$, 
see \cite{ammann.grosjean.humbert.morel:p07} for details.

If Inequality (\ref{aubin1}) holds strictly then one can show that the
infimum in Equation (\ref{funct}) is attained by a spinor field $\phi$. 
The following theorem will be a central ingredient in the proof of
Theorem \ref{main}. 
\begin{theorem}[\cite{ammann:p03, ammann:habil}] 
\label{attained}
Let $(M,g)$ be a compact spin manifold of dimension $n$ for which 
Inequality (\ref{aubin1}) holds strictly. Then there exists a spinor
field  
$\phi \in C^{2,\alpha}(\Si M) \cap 
C^{\infty}(\Si M \setminus \phi^{-1}(0))$ where 
$\al\in(0,1)\cap (0,2/(n-1)]$ such that
$\| \phi \|_{L^\frac{2n}{n-1}(M)} = 1$ and
$$ 
D\phi
= 
\lamin(M,g) |\phi|^{\frac{2}{n-1}} \phi.
$$
\end{theorem}
Furthermore the infimum in the definition of $\lamin(M,g)$ is
attained by the generalized conformal metric $\ti g=|\phi|^{4/(n-1)}g$, 
see \cite{ammann:p03} for details.

%%%%%%%%%%%%%%%%%%%%%%%%%%%%%%%%%%%%%%%%%%%%%%%%%%%%%%%%%%%%%%
\section{Preparations for proofs}
%%%%%%%%%%%%%%%%%%%%%%%%%%%%%%%%%%%%%%%%%%%%%%%%%%%%%%%%%%%%%%

%%%%%%%%%%%%%%%%%%%%%%%%%%%%%%%%%%%%%%%%%%%%%%%%%%%%%%%%%%%%%%
\subsection{Removal of singularities}
%%%%%%%%%%%%%%%%%%%%%%%%%%%%%%%%%%%%%%%%%%%%%%%%%%%%%%%%%%%%%%

The following theorem gives a condition for when singularities of
solutions to Dirac equations can be removed. 
\begin{theorem} \label{thm.removal}
Let $(V,g)$ be a (not necessarily complete) Riemannian spin manifold 
and let $S$ be a compact submanifold of $V$ of codimension 
$m \geq 2$. Assume that $\phi\in L^p(\Si(V\setminus S))$, 
$p\geq m/(m-1)$, satisfies the equation
$$
D\phi = \rho
$$
weakly on $V\setminus S$ where 
$\rho\in L^1(\Si(V\setminus S)) = L^1(\Si V)$. Then this equation
holds weakly on $V$. In particular the singular support of the 
distribution $D\phi$ is empty. 
\end{theorem}

\begin{proof}
Let $\psi$ be a smooth compactly supported spinor. We have to show
that  
\begin{equation} \label{eqn.thm.removal}
\int_V \<\phi,D \psi\> \,dv =  \int_V \< \rho,\psi\> \,dv.
\end{equation}
Recall that for $\ep>0$ we denote the set of points in $V$ of
distance less than $\ep$ to $S$ by $U(S,\ep)$. 
We choose a smooth cut-off function
$\chi_\ep:V\to[0,1]$ with support in 
$U(S,2\ep)$, $\chi_\ep = 1$ on $U(S,\ep)$, and 
$|\grad \chi_\ep| \leq 2/\ep$. We then have 
\begin{equation*}
\begin{split}
\int_V \<\phi,D \psi\> \,dv - \int_V \<\rho,\psi\> \,dv 
&=
\int_V \< \phi,D ( (1-\chi_\ep) \psi + \chi_\ep \psi) \> \,dv
-
\int_V \< \rho,\psi \> \,dv\\
&=
\int_V \< D\phi , (1-\chi_\ep) \psi \> \,dv
+\int_V \< \phi,\chi_\ep D \psi \> \,dv \\
&\quad 
+\int_V\<\phi,\grad \chi_\ep \cdot \psi\>\,dv 
-\int_V\<\rho,\psi\> \,dv \\
&=  
-\int_V \<\rho,\chi_\ep \psi \> \,dv
+\int_V \< \phi,\chi_\ep D \psi\>\,dv \\
&\quad
+\int_V \<\phi,\grad \chi_\ep \cdot \psi\> \,dv, 
\end{split} 
\end{equation*}
where $D\phi = \rho$ is used in the last equality. Let $q$ be related
to $p$ via $1/q+1/p=1$. It follows that
\begin{equation*}
\begin{split}
\left| \int_V \<\phi,D \psi\> \,dv
- \int_V \<\rho,\psi\> \,dv \right|
&\leq  
\left( \sup_{U(S,2\ep)} |\psi|\right) 
\int_{U(S,2\ep)} |\rho| \, dv \\
&\quad
+
\left( \sup_{U(S,2\ep)} |D\psi| +
\frac{2}{\ep} \sup_{U(S,2\ep)} |\psi| \right) 
\int_{U(S,2\ep)} |\phi| \,dv  \\
&\leq 
o(1) + 
\frac{C}{\ep} \|\phi\|_{L^p( U(S,2\ep) )} 
\vol(U(S,2\ep))^{1/q} \\
&\leq 
o(1) +  
C  \|\phi\|_{L^p( U(S,2\ep) )} \ep^{(m/q) - 1},
\end{split}
\end{equation*}
where $o(1)$ denotes a term tending to $0$ as $\ep\to 0$. Since 
$p \geq m/(m-1)$ is equivalent to $m/q \geq 1$ we see that 
(\ref{eqn.thm.removal}) holds.
\end{proof}
Applying Theorem \ref{thm.removal} to the non-linear Dirac equation in
Theorem \ref{attained} we get the following corollary.  
\begin{corollary} \label{removal}
Let $V$ and $S$ be as in Theorem \ref{thm.removal}. Then any
$L^p$-solution, $p=2n/(n-1)$, of
\begin{equation} \label{eq.singsol}
D\phi = \la |\phi|^{p-2} \phi
\end{equation}
on $V\setminus S$ is also a weak $L^p$-solution of (\ref{eq.singsol})
on $V$.
\end{corollary}

%%%%%%%%%%%%%%%%%%%%%%%%%%%%%%%%%%%%%%%%%%%%%%%%%%%%%%%%%%%%%%
\subsection{Limit spaces and limit solutions}
%%%%%%%%%%%%%%%%%%%%%%%%%%%%%%%%%%%%%%%%%%%%%%%%%%%%%%%%%%%%%%

In the proofs of the main theorems we will construct limit solutions
of a Dirac equation on certain limit spaces. For this we need the following
two lemmas. In the statement of these results, in order to simplify the notations, we write $\al \to 0$
instead of $\al_i \to 0$ as $i \to \infty$ when $(\al_i)_{i \in \mN}$ is a
sequence of positive numbers converging to $0$. In the same way, the
subsequences of $(\al_i)$ will also be denoted by $(\al)$.

\begin{lemma} \label{diffeom}
Let $V$ be an $n$-dimensional manifold. Let $(p_{\al})$
be a sequence of points in $V$ which converges to a point $p$ as 
$\al \to 0$. Let $(\ga_\al)$ be a sequence of metrics defined on a 
neighbourhood $O$ of $p$ which converges to a metric $\ga_0$ in 
the $C^2(O)$-topology. Finally, let $(b_\al)$ be a sequence of
positive real numbers such that $\lim_{\al \to 0} b_\al = \infty$. 
Then for $r>0$ there exists for $\alpha$ small enough a diffeomorphism  
$$
\Theta_\al: 
B^n(r)
\to
B^{\ga_\al} (p_\al, b_\al^{-1} r)
$$
with $\Theta_\al(0)= p_\al$ such that the metric 
$\Theta_{\al}^*(b_\al^2 \ga_\al)$ tends to the  
Euclidean metric $\xi^n$ in $C^1(B^n(r))$. 
\end{lemma}

\begin{proof}
Denote by $\exp^{\ga_\al}_{p_\al}: U_\al \to O_\al$ the exponential
map at the point $p_\al$ defined with respect to the metric
$\ga_\al$. Here $O_\al$ is a neighbourhood of $p_\al$ in $V$ and
$U_{\al}$ is a neighbourhood of the origin in $\mR^n$. We set 
$$ 
\Theta_\al:  
B^n(r) \ni x
\mapsto
\exp^{\ga_\al}_{p_\al}( b_\al^{-1} x) \in
B^{\ga_\al} (p_\al, b_\al^{-1} r) .
$$
It is easily checked that $\Theta_\al$ is the desired diffeomorphism.
\end{proof}

\begin{lemma}\label{lim_sol}
Let $V$ an $n$-dimensional spin manifold. Let $(g_\al)$ be a sequence 
of metrics which converges to a metric $g$ in $C^1$ on all compact
sets $K \subset V$ as $\al \to 0$. Assume that $(U_\al)$ is an
increasing sequence of subdomains of $V$ such that 
$\cup_{\al} U_\al = V$. Let 
$\psi_\al \in \Gamma(\Sigma^{g_\al} U_\al)$ be a sequence of 
spinors of class $C^1$ such that 
$\| \psi_{\al} \|_{L^\infty(U_\al)} \leq C$ where $C$ does 
not depend on $\al$, and 
\begin{equation} \label{eqal}
D^{g_{\al}} \psi_\al = \la_\al |\psi_\al|^{\frac{2}{n-1}} \psi_\al
\end{equation}
where the $\la_\al$ are positive numbers which tend to 
$\bar{\la} \geq 0$. Then there exists a spinor 
$\psi \in \Gamma(\Sigma^g V)$ of class $C^1$ such that 
\begin{equation} \label{eq_limit}
D^g \psi 
= 
\bar{\la} |\psi|^{\frac{2}{n-1}} \psi
\end{equation} 
on $V$ and a subsequence of $(\beta_g^{g_\al} \psi_\al )$ 
tends to $\psi$ in $C^0(K)$  for any compact set $K \subset {V}$.
In particular
\begin{equation} \label{norminf_lim}
\| \psi \|_{L^\infty(K)} 
= 
\lim_{\al \to 0} \| \psi_\al \|_{L^\infty(K)},
\end{equation} 
and 
\begin{equation} \label{normlr_lim}
\int_K |\psi|^r \,dv^g 
= 
\lim_{\al \to 0} \int_K |\psi_\al|^r \,dv^{g_\al}
\end{equation}
for any compact set $K$ and any $r \geq 1$.
\end{lemma} 

\begin{proof}
Let $K$ be a compact subset of $V$ and let $\Om$ be an open set in 
$V$ with compact closure such that $K \subset \Om$. Let 
$\chi \in C^{\infty}(V)$ with $0 \leq \chi \leq 1$ be compactly 
supported in $\Om$ and satisfy $\chi = 1$ on a neighbourhood 
$\widetilde{\Om}$ of $K$. Set 
$\phi_\al =  ( \beta_{g_\al}^g)^{-1} \psi_\al$. Using Equations 
(\ref{eqal}) and (\ref{relD}) we get
\begin{equation} \label{dphial}
D^g (\chi \phi_\al) 
= 
\grad^g \chi \cdot \phi_\al 
+ \chi\la_{\al} |\phi_\al|^{\frac{2}{n-1}} \phi_\al 
- \chi A^g_{g_\al} (\nabla^g \phi_{\al} ) 
- \chi B^g_{g_\al} (\phi_{\al} ).
\end{equation}
Using the fact that $|a+b+c|^r\leq 3^r(|a|^r+|b|^r+|c|^r)$ for  
$a,b,c \in \mR$, $r \geq 1$, we see that 
\begin{equation*}
\begin{split}
| D^g( \chi \phi_\al)|^r 
&\leq 
3^r \Big(
| \grad^g \chi \cdot \phi_\al + 
\chi \la_{\al} |\phi_\al|^{\frac{2}{n-1}} \phi_\al |^r   \\
&\qquad
+ 
|\chi A^g_{g_\al} (\nabla^g \phi_{\al} )|^r 
+ 
|\chi B^g_{g_\al} (\phi_{\al} )|^r \Big) 
\end{split}
\end{equation*}
for  $r \geq 1$. Since $\| \phi_{\al} \|_{L^\infty(V)} 
= \| \psi_{\al} \|_{L^\infty(V)} \leq C$ we have 
$$
| \grad^g \chi \cdot \phi_\al + 
\chi \la_{\al} |\phi_\al|^{\frac{2}{n-1}} \phi_\al |^r
\leq C.
$$
By Relations (\ref{boundA^g_g'}) and (\ref{boundB^g_g'}), and 
since $\lim_{\al \to 0} \| g_\al - g \|_{C^1(\Om)} =0$, we get 
\begin{equation*}
\begin{split}
|\chi A^g_{g_\al} (\nabla^g \phi_{\al} )|^r 
+ 
| \chi B^g_{g_\al} (\phi_{\al} )|^r 
&\leq 
o(1) \left( 
|\nabla^g (\chi \phi_\al) |^r 
+ |\grad^g \chi \cdot \phi_\al |^r
+ |\chi \phi_\al |^r
\right) \\
&\leq
o(1) \left( |\nabla^g (\chi \phi_\al) |^r  + C \right) ,
\end{split}
\end{equation*}
where $o(1)$ tends to $0$ with $\al$. It follows that 
$$
|D^g (\chi \phi_\al)|^r \leq C + o(1) |\nabla^g (\chi \phi_\al) |^r.
$$
Setting $\phi = \chi \phi_\al$ in Inequality (\ref{ellip_est}) 
and again using that $\| \phi_{\al} \|_{L^\infty(\Omega)}$ is
uniformly bounded we get that
$$
\int_{\Omega} |\nabla^g (\chi \phi_\al) |^r \,dv^g 
\leq 
C + o(1) \int_{\Omega} |\nabla^g (\chi \phi_\al) |^r \,dv^g.
$$
In particular $(\chi \phi_\al)$ is bounded in $H_0^{1,r}(\Om)$. 
Let $a \in (0,1)$. By the Sobolev Embedding Theorem this implies that
a subsequence of $(\chi \phi_\al)$ converges in $C^{0,a}(\Omega)$ to
$\psi_K \in \Gamma(\Sigma^{g_\al} \Omega)$ of class $C^{0,a}$. We take
the inner product of (\ref{dphial}) with a smooth spinor 
$\widetilde{\phi}$ which is compactly supported in $\widetilde{\Om}$
and integrate over $\Omega$. Since $\chi = 1$ on the support of
$\widetilde{\phi}$ the result is
\begin{equation*}
\begin{split}
\int_\Omega \< \phi_{\al} ,D^g \widetilde{\phi} \> \,dv^g 
&=  
\int_\Omega \la_{\al} 
|\phi_\al|^{\frac{2}{n-1}} \< \phi_\al,\widetilde{\phi} \> \,dv^g  \\
&\quad
- \int_\Omega 
\<A^g_{g_\al}(\nabla^g \phi_{\al}  ), \widetilde{\phi} \> \,dv^g 
- \int_\Omega 
\<B^g_{g_\al}(\phi_{\al} ),\widetilde{\phi} \> \,dv^g .
\end{split}
\end{equation*}
Taking the limit $\al \to 0$ and again using (\ref{boundA^g_g'}) and 
(\ref{boundB^g_g'}) we get  
$$
\int_\Omega \< \psi_K ,D^g \widetilde{\phi} \>\,dv^g 
= 
\int_\Omega \bar{\la} |\psi_K|^{\frac{2}{n-1}} 
\< \psi_K,\widetilde{\phi} \> \,dv^g.
$$
Hence, $\psi_K$ satisfies Equation (\ref{eq_limit}) weakly on $K$.
By standard regularity theorems we conclude that $\psi_K \in C^1(K)$. 

Now we choose an increasing sequence of compact sets $K_m$ such that
$\cup_m K_m = V$. Using the above arguments and taking successive
subsequences it follows that $(\phi_\al)$ converge to spinor fields
$\psi_m$ on $K_m$ with $\psi_m |_{K_{m-1}} = \psi_{m-1}$.  We define
$\psi$ on $V$ by $\psi \definedas \psi_m$ on $K_m$. By taking a
diagonal subsequence of we get that $(\phi_\al)$ tends to $\psi$ in
$C^0$ on any compact set $K \subset V$.

The relations (\ref{norminf_lim}) and (\ref{normlr_lim}) follow
immediately since $\beta^g_{g_\al}$ is an isometry, since 
$\phi_\al = (\beta^g_{g_\al})^{-1} \psi_\al$, and since $(g_\al)$
(resp. $(\phi_\al)$) tends to $g$ (resp. $\psi$) in $C^0$ on $K$. 
This ends the proof of Lemma \ref{lim_sol}.
\end{proof}

%%%%%%%%%%%%%%%%%%%%%%%%%%%%%%%%%%%%%%%%%%%%%%%%%%%%%%%%%%%%%%
\subsection{Dirac spectral bounds on products with spheres}
\label{sec.lemmata}
%%%%%%%%%%%%%%%%%%%%%%%%%%%%%%%%%%%%%%%%%%%%%%%%%%%%%%%%%%%%%%

In the following lemma we assume (in the case $m=1$) that $S^1$
carries the spin structure which is obtained by restricting the unique
spin structure on the $B^2$ to the boundary. The proof is
a simple application of the formula for the squared Dirac operator on
a product manifold together with the lower bound of its spectrum on
the standard sphere.

\begin{lemma} \label{sphere_bundle}
Let $(V,g)$ be a complete Riemannian spin manifold. Then any
$L^2$-spinor $\psi$ on $(V \times S^m, g + \si^m)$ satisfies 
$$
\int_{V \times S^m} |D\psi|^2 \,dv^{g + \si^m} 
\geq 
\frac{m^2}{4} \int_{V \times S^m} |\psi|^2 \,dv^{g + \si^m} .
$$
\end{lemma}

%%%%%%%%%%%%%%%%%%%%%%%%%%%%%%%%%%%%%%%%%%%%%%%%%%%%%%%%%%%%%%
\subsection{Approximation by local product metrics} 
\label{sec.approx}
%%%%%%%%%%%%%%%%%%%%%%%%%%%%%%%%%%%%%%%%%%%%%%%%%%%%%%%%%%%%%%

In this paragraph we will see how to change the metrics $g_i$ to
product form $g_i = h_i + dr_i^2 + r_i^2 \sigma^{n-k-1}$ in a
neighbourhood of $W_i'$ in $M_i$ without changing $\lamin (M_i,g_i)$
much.

\begin{lemma} \label{approx}
Let $(V,g)$ be a compact Riemannian manifold of dimension~$n$ and 
let $S$ be a closed submanifold of dimension~$k$, where 
$0 \leq k \leq n-2$. Assume that a trivialization of the normal bundle
of $S$ is given and assume that $D^g$ is invertible. Then there exists
a sequence $(\ep_i)_{i \in \mN}$ of positive real numbers converging to $0$
and a sequence  $(g_{\ep_i})$ of metrics on $V$ such that 
$$
\lim_{i \to \infty} \lamin(V,g_{\ep_i}) = \lamin (V,g)
$$
and
$$
g_{\ep_i} = h + dr^2 + r^2 \sigma^{n-k-1}
$$
on $U^g(S,\ep_i)$. Here $h$ is the restriction of the metric $g$ to
$S$ and $r(x) = d^g(S,x)$.  
\end{lemma}

\begin{proof}
Using the trivialization of the normal bundle we identify a
neighbourhood of $S$ with $S \times B^{n-k}(\Rmax)$ as described in
Section \ref{joining}. In this neighbourhood we define the metric 
$\overline{g} \definedas h + dr^2 + r^2 \sigma^{n-k-1}$.
Recall that $U^g(S,\ep)$ denotes the set of points $x \in V$ such
that $r(x) < \ep$ and let $\chi_\ep \in C^{\infty}(M)$, 
$0 \leq \chi \leq 1$, be a cut-off function such that $\chi = 1$ on 
$U^g(S,\ep)$, $\chi = 0$ on $M \setminus U^g(S,2\ep)$, and 
$|d \chi_{\ep}|\leq 2/\ep$. We define
$$
g_\ep
\definedas
\chi_\ep \overline{g}+ (1-\chi_\ep) g.
$$
Then $g_\ep$ has product form on $U^g(S,\ep)$. For convenience we
introduce the notation $\la_\ep \definedas \lamin(V,g_\ep)$ and 
$\la \definedas \lamin(V,g)$. Let $(\ep_i)_{\i \in \mN}$ be a sequence of
positive numbers tending to $0$ such
that  the limit $\lim_{i \to \infty} \la_{\ep_i}$ exists. In the following, we
write $\ep \to 0$ instead of
$\ep_i \to 0$ as $i \to \infty$. In the same way, $(\ep)$ will denote
the successive subsequences of $(\ep_i)$ we will need. With this notations,
let $\bar{\la} \definedas \lim_{\ep \to 0} \la_\ep$ which exists after 
possibly taking a subsequence.  

We begin by proving that 
\begin{equation} \label{la_al<la}
\bar{\la} \leq \la,
\end{equation}
which is the simpler part of the proof. Let $J \definedas J^g$ and 
$J_\ep \definedas J^{g_{\ep}}$ be the functionals associated to $g$
and $g_{\ep}$, and let $\de>0$ be a small number. We set 
$\chi_\ep' \definedas 1-\chi_{2\ep}$, so that 
$\chi_\ep' = 1$ on $V \setminus U^g(S,4\ep)$, $\chi_\ep' = 0$ on 
$U^g(S,2\ep)$, and $|d\chi_\ep'| \leq 1/\ep$. We see that $g=g_\ep$ 
on the support of $\eta_\ep'$. Let $\psi$ be a smooth spinor such 
that $J(\psi) \leq \la + \de$. We then have 
$$
\int_V \< D^g (\chi_\ep' \psi), \chi_\ep' \psi \> \,dv^g 
= 
\int_V \chi_\ep'^2 \< D^g \psi, \psi\> \,dv^g 
+ 
\int_V \< \grad^g \chi_\ep' \cdot \psi, \chi_\ep' \psi \> \,dv^g.
$$
Since the last term here is purely imaginary we obtain
\begin{equation} \label{limden}
\lim_{\ep \to 0} 
\int_V  \< D^g (\chi_\ep' \psi), \chi_\ep' \psi \> \,dv^g  
= 
\lim_{\ep \to 0}
\operatorname{Re}
\int_V \chi_\ep'^2 \< D^g \psi, \psi\> \,dv^g
=
\int_V \<D^g \psi, \psi \> \,dv^g.
\end{equation}
We compute
\begin{equation} \label{limnum} 
\begin{split}
\int_V |D^g( \chi_\ep' \psi)|^{\frac{2n}{n+1}} \,dv^g
&=
\int_{V \setminus U^g(S,4\ep)} |D^g\psi|^{\frac{2n}{n+1}} \,dv^g \\
&\quad
+ 
\int_{U^g(S,4\ep) \setminus U^g(S,2\ep)} 
|\grad^g \chi_\ep' \cdot \psi 
+ \chi_\ep' D^g \psi |^{\frac{2n}{n+1}} \,dv^g.
\end{split}
\end{equation}
Using the fact that $|a+b|^{\frac{2n}{n+1}} \leq
2^{\frac{2n}{n+1}}(|a|^{\frac{2n}{n+1}} +|b|^{\frac{2n}{n+1}})$
for $a,b \in \mR$ we have 
\begin{equation*}
\begin{split}
|\grad^g \chi_\ep' \cdot \psi + \chi_\ep' D^g\psi |^{\frac{2n}{n+1}} 
&\leq
2^\frac{2n}{n+1} \left(  
|\grad^g \chi_\ep'|^{\frac{2n}{n+1}} |\psi|^{\frac{2n}{n+1}} 
+  
| \chi_\ep'|^{\frac{2n}{n+1}} |D^g\psi|^{\frac{2n}{n+1}} 
\right) \\
&\leq 
2^\frac{2n}{n+1} \left( C_1 \ep^{-\frac{2n}{n+1}} + C_2 \right),
\end{split} 
\end{equation*}
where $C_1$ and $C_2$ are bounds on $|\psi|$ and $|D\psi|$. Since 
$\Vol( U^g(S,4\ep) \setminus U^g(S,2\ep)) 
\leq C \ep^{n-k} \leq C \ep^2$
it follows that
$$
\lim_{\ep \to 0}
\int_{U^g(S,4\ep) \setminus U^g(S,2\ep)} 
| \grad^g \chi_\ep' \cdot \psi 
+ \chi_\ep' D^g\psi |^{\frac{2n}{n+1}} \,dv^g
= 0.
$$
It is clear that 
$\lim_{\ep \to 0} \int_{V \setminus U^g(S,4\ep)} 
|D^g\psi|^{\frac{2n}{n+1}} \,dv^g
= \int_V  |D^g\psi|^{\frac{2n}{n+1}} \,dv^g$
so Equation (\ref{limnum}) tells us that 
$$
\lim_{\ep \to 0} 
\int_V |D^g( \chi_\ep' \psi)|^{\frac{2n}{n+1}} \,dv^g 
= 
\int_V  |D^g\psi|^{\frac{2n}{n+1}} \,dv^g.
$$
Together with Equation (\ref{limden}) this proves that 
$\lim_{\ep \to 0} J(\chi_\ep' \psi) = J(\psi) \leq \la +\de$. 
Since $g_\ep = g$ on the support of $\chi_\ep' \psi$, we have 
$J_{\ep} (\chi_\ep' \psi)= J(\chi_\ep' \psi)$.
Relation (\ref{la_al<la}) now follows since
$\la_\ep \leq J_{\ep} (\chi_\ep' \psi)$ and $\de$ is arbitrary.

The second and harder part of the proof is to show that
\begin{equation} \label{la_al>la}
\bar{\la} \geq \la.
\end{equation}

From Proposition \ref{aubin} we know that
$\la_\ep \leq  \lamin(S^n,\si^n)$,
$\bar{\la} \leq  \lamin(S^n,\si^n)$, and $\la \leq  \lamin(S^n,\si^n)$.  
Inequality (\ref{la_al>la}) is obvious if 
$\bar{\la} =  \lamin(S^n,\si^n)$. Hence we will assume  
$\la_\ep < \lamin(S^n,\si^n)$ for a sequence $\ep\to 0$.
As the Dirac operator is invertible we know that \eref{ellip_est2}
holds. By Theorem \ref{attained} there exists for all $\ep$
spinor fields $\psi_\ep \in \Gamma(\Si^{g_\ep} v)$ of class $C^1$ 
such that  
\begin{equation} \label{eq-alpha}
D^{g_{\ep}} \psi_\ep 
= 
\la_\ep |\psi_\ep|^{\frac{2}{n-1}} \psi_\ep,
\end{equation}
and
\begin{equation} \label{normalize}
\int_V |\psi_\ep|^{\frac{2n}{n-1}} \,dv^{g_\ep} = 1.
\end{equation}
Define $\phi_\ep = (\beta_{g_\ep}^g)^{-1} \psi_\ep$. 
Since $g_\ep \to g$ it is easily seen that the sequence 
$(\phi_\ep)$ is bounded in $L^\frac{2n}{n-1}(V,g)$. By (\ref{relD}) 
and (\ref{eq-alpha}) we have
\begin{equation} \label{dphial1}
D^g \phi_\ep 
= 
\la_\ep |\phi_\ep|^{\frac{2}{n-1}} \phi_\ep
- A^g_{g_\ep} (\nabla^g  \phi_{\ep} ) 
- B^g_{g_\ep} (\phi_{\ep} ),
\end{equation} 
together with $|a+b+c|^{\frac{2n}{n+1}} \leq 3^{\frac{2n}{n+1}}
(|a|^{\frac{2n}{n+1}}+|b|^{\frac{2n}{n+1}}+|c|^{\frac{2n}{n+1}})$ 
for $a,b,c \in \mR$ this implies
\begin{equation} \label{dphial1der}
| D^g \phi_\ep|^{\frac{2n}{n+1}}
\leq 
C \left( \la_{\ep}^{\frac{2n}{n+1}} |\phi_\ep|^{\frac{2n}{n-1}}
+ 
|A^g_{g_\ep} (\nabla^g \phi_{\ep} )|^{\frac{2n}{n+1}}
+ 
|B^g_{g_\ep} ( \phi_{\ep} )|^{\frac{2n}{n+1}} \right). 
\end{equation}
We also have 
\begin{equation} \label{Ag} 
| A^g_{g_\ep} (\nabla^g \phi_{\ep} )| 
\leq 
\| g - g_\ep \|_{C^0(V)} |\nabla^g \phi_{\ep}| 
\leq 
C \ep |\nabla^g \phi_{\ep}| ,
\end{equation}
and 
\begin{equation} \label{Bg}
| B^g_{g_\ep} ( \phi_{\ep} )| 
\leq 
\| g - g_\ep \|_{C^1(V)} |\phi_{\ep}| 
\leq 
C |\phi_{\ep}|. 
\end{equation}
Indeed, since $g$ and $g_{\ep}$ coincide on $S$, there exists a
constant $C$ so that $\| g - g_{\ep} \|_{ B^g(V,\ep) } \leq C \ep$. 
Together with the fact that $|d \chi_\ep| \leq 2/\ep$ and using 
the definition of $g_\ep$, this immediately implies that 
$\| g- g_\ep \|_{C^1(V)} \leq C$. Using Relation (\ref{ellip_est2})
and integrating (\ref{dphial1der}) we find that 
\begin{equation*} 
\begin{split}
\int_V |\nabla^g \phi_\ep|^{\frac{2n}{n+1}} \,dv^g 
&\leq C \Big( 
\la_\ep^{\frac{2n}{n+1}} \int_V |\phi_\ep|^{\frac{2n}{n-1}} \,dv^g 
+ 
\ep^{\frac{2n}{n+1}} \int_V 
|\nabla^g \phi_\ep|^{\frac{2n}{n+1}}\,dv^g \\
&\qquad
+ 
\int_V |B^g_{g_\ep}(\phi_\ep)|^{\frac{2n}{n+1}}\,dv^g
\Big).
\end{split}
\end{equation*}
As $g$ and $g_\ep$ coincide on $V \setminus B^g(S,2\ep)$ we conclude
that $B^g_{g_\ep} ( \phi_{\ep} ) = 0$ on this set. Together with 
(\ref{Bg}) we have 
\begin{equation*} 
\begin{split}
\int_V  | B^g_{g_\ep}  (\phi_{\ep} )|^{\frac{2n}{n+1}} \,dv^g 
&\leq   
C \int_{ B^g(S,2\ep)}  |\phi_{\ep} |^{\frac{2n}{n+1}} \,dv^g \\
&\leq  
C \, \Vol(B^g(S,2\ep))^{\frac{2}{n+1}} 
{\left( \int_{B^g(S,2\ep)} |\phi_{\ep}|^{\frac{2n}{n-1}}\,dv^g 
\right)}^{\frac{n-1}{n+1}} \\
&= o(1) ,
\end{split}
\end{equation*}
where $o(1)$ tends to $0$ with $\ep$. Hence
\begin{equation} \label{nabla_ineq} 
\int_V |\nabla^g \phi_\ep|^{\frac{2n}{n+1}} \,dv^g 
\leq C \left( 
\la_\ep^{\frac{2n}{n+1}} \int_V |\phi_\ep|^{\frac{2n}{n-1}} \,dv^g 
+ 
\ep \int_V |\nabla^g \phi_\ep|^{\frac{2n}{n+1}}\,dv^g 
+ 
o(1) 
\right) .
\end{equation}
This implies in particular that $(\phi_\ep)$ is bounded in 
$H_1^{\frac{2n}{n+1}}(V)$ and hence after passing to a subsequence
$(\phi_\ep)$ converges weakly to a limit $\phi$ in
$H_1^{\frac{2n}{n+1}}(V)$. 

The next step is to prove that $\bar{\la} = \lim_{\ep \to 0} \la_{\ep}$ 
is not zero. To get a contradiction let us assume that $\bar{\la} = 0$.
We then obtain from (\ref{nabla_ineq}) that 
$$
\int_V |\nabla^g \phi|^{\frac{2n}{n+1}}\,dv^g 
\leq
\lim_{\ep \to 0}  \int_V |\nabla^g \phi_\ep|^{\frac{2n}{n+1}}\,dv^g 
=0.
$$
So $\phi$ is parallel and since $D^g$ is invertible we conclude 
$\phi = 0$, in other words $(\phi_\ep)$ converges weakly to zero in
$H_1^{\frac{2n}{n+1}} (V)$. As this space embeds compactly into 
$L^{\frac{2n}{n+1}} (V)$ we have
$$
\lim_{\ep \to 0} \| \phi_\ep \|_{L^{\frac{2n}{n+1}}(V)} 
= 
\| \phi \|_{L^{\frac{2n}{n+1}}(V)} 
=
0 ,
$$
and hence $(\phi_\ep)$ converges strongly to zero in
$H_1^{\frac{2n}{n+1}}(V)$. As this space embeds continuously into 
$L^{\frac{2n}{n-1}}(V)$ we conclude that the sequence
converges strongly to zero in $L^{\frac{2n}{n-1}}(V)$. This is 
impossible since by Relation (\ref{normalize}) we easily get that 
$$
\lim_{\ep \to 0} \| \phi_\ep \|_{L^{\frac{2n}{n-1}}(V)} = 1.
$$
From this contradiction we conclude
\begin{equation} \label{la>0}
\bar{\la} > 0.
\end{equation}
From (\ref{dphial1}) we have  
\begin{equation*}
\begin{split}
{\| D^g \phi_\ep \|}_{L^\frac{2n}{n+1}(V)} 
&\leq 
\la_\ep {\| \phi_\ep \|}^\frac{n-1}{n+1}_{L^\frac{2n}{n-1}(V)} 
+
{\| A^g_{g_\ep}(\nabla^g \phi_{\ep}) \|}_{L^\frac{2n}{n+1}(V)} 
\\
&\quad
+ 
{\| B^g_{g_\ep}(\phi_{\ep} )\|}_{L^\frac{2n}{n+1}(V)}.
\end{split}
\end{equation*}
We already proved above that 
$$
\lim_{\ep \to 0} {\| B^g_{g_\ep} ( \phi_{\ep} )
\|}_{L^\frac{2n}{n+1}(V)} = 0.
$$
Using Relation (\ref{Ag}) we get similarily
$$
\lim_{\ep \to 0} {\| A^g_{g_\ep} (\nabla^g \phi_{\ep} ) 
\|}_{L^\frac{2n}{n+1}(V)} = 0.
$$
Moreover since $dv^{g_{\ep}} = (1 + o(1))\,dv^g$ it follows from 
(\ref{normalize}) that 
$$
\la_\ep \| \phi_\ep \|^{\frac{n+1}{n-1}}_{L^\frac{2n}{n-1}(V)} 
= 
\la_\ep (1+o(1)).
$$
We conclude
\begin{equation} \label{num2} 
{\| D^g \phi_\ep \|}_{L^\frac{2n}{n+1}(V)}
\leq \la_\ep  + o(1).
\end{equation}
Starting from Equation (\ref{dphial1}) we can prove in a 
similar way that
\begin{equation} \label{den2}
\int_V \< D^g \phi_\ep , \phi_\ep  \> \,dv^g 
\geq 
\la_\ep + o(1).
\end{equation}
From (\ref{la>0}), (\ref{num2}), and (\ref{den2}) it follows that 
$\la \leq \lim_{\ep \to 0} J(\phi_\ep)= \bar{\la}$. This ends
the demonstration of (\ref{la_al>la}), which together with 
(\ref{la_al<la}) proves Lemma \ref{approx}.
\end{proof}

%%%%%%%%%%%%%%%%%%%%%%%%%%%%%%%%%%%%%%%%%%%%%%%%%%%%%%%%%%%%%%
\section{Proofs}
%%%%%%%%%%%%%%%%%%%%%%%%%%%%%%%%%%%%%%%%%%%%%%%%%%%%%%%%%%%%%%

%%%%%%%%%%%%%%%%%%%%%%%%%%%%%%%%%%%%%%%%%%%%%%%%%%%%%%%%%%%%%%
\subsection{Proof of Theorem \ref{main}}
%%%%%%%%%%%%%%%%%%%%%%%%%%%%%%%%%%%%%%%%%%%%%%%%%%%%%%%%%%%%%%

This section is devoted to the proof of Theorem \ref{main}. 
Our goal is to construct a family of 
metrics $(g_\th)$ with $0 <\th < \th_0$ which satisfies the conclusion of Theorem 
(\ref{main}).

From Lemma \ref{approx} applied with $V = M = M_1\amalg M_2$ and 
$S = W'= w_1(W \times \{0\}) \amalg w_2(W \times \{0\})$ 
we may assume that 
\begin{equation} \label{metric=product} 
g = h + dr^2 + r^2 \sigma^{n-k-1}
\end{equation} 
in a neighbourhood $U(\Rmax)$ of $W'$ where $\Rmax>0$. 
We fix numbers $R_0,R_1\in \mR$ with $\Rmax> R_1> R_0 >0$  
and we choose a function $F:M\setminus W' \to \mR^+$ such that 
$$
F(x) = 
\begin{cases}
1,       &\text{if $x \in M_i \setminus U_i(R_1)$;} \\ 
r_i^{-1} &\text{if $x \in U_i(R_0)\setminus W'$.}
\end{cases}
$$
We further choose $\theta\in (0,R_0)$, later we will let 
$\theta\to 0$.
It is not difficult to see that there is a smooth function
$f:U(\Rmax) \to \mR$ (depending only on $r$), 
real numbers $\de_1 = \de_1(\th)$ and $ \de_2 = \de_2(\th)$ with 
$\th> \de_2 > \de_1>0$ and a real number 
$A_\th\in (\th^{-1},\de_2^{-1})$ such that
%%%%%%%%%%%%%%%%%%%%%%%%%%%%%%%%%%%%%%%%
\begin{figure}
\newdimen\axdim
\axdim=1pt
\newdimen\cudim
\cudim=2pt
\def\ticklen{.2}
\def\abst{.3}  

\begin{center}
\psset{unit=1cm}

\begin{pspicture}(-1,-\ticklen)(10.5,7)
\psset{linewidth=\axdim}
\psaxes[linewidth=\axdim,labels=none,ticks=none]{->}(0,0)(9.5,6)
\rput[t](9.5,-\abst){$-\ln r=|\ln \epsilon|-|t|$} 
\rput[r](-\abst,6){$f(r)$}  
\psline(3.0,-\ticklen)(3.0,\ticklen)
\psline(5.0,-\ticklen)(5.0,\ticklen)
\rput[t](3.0,-\abst){$-\ln\theta$}
\rput[t](5.0,-\abst){$-\ln \delta_2$}
\psline(-\ticklen,4)(\ticklen,4)
\rput[r](-\abst,4){$\ln A_\theta$}
\psline[linestyle=dotted]{-}(3.0,0)(3,3)
\psline[linestyle=dotted]{-}(5,0)(5,4)
\psline[linestyle=dotted]{-}(0,4)(5,4)
\psset{linewidth=\cudim}

\psline(0,0)(3,3)
\psecurve[showpoints=false](2.8,2.7)(3,3)(5,4)(5.2,3.9)
\psline(5,4)(7,4)
\end{pspicture}
\end{center}
\smallskip
\caption{The function $-\ln r \mapsto f(r)$}
\end{figure}
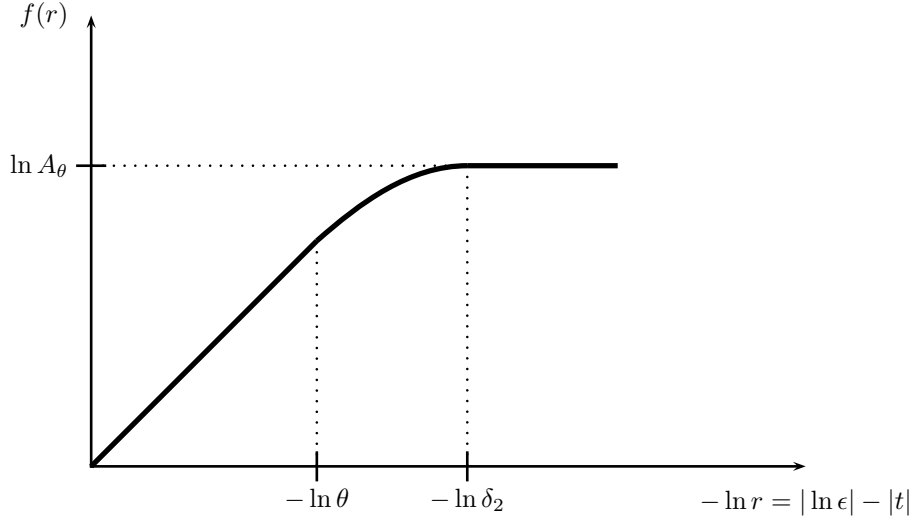
%%%%%%%%%%%%%%%%%%%%%%%%%%%%%%%%%%%%%%%% 
$$
f(x)  =  
\begin{cases}
-\ln r     &\text{if $x \in U(\Rmax) \setminus U(\th)$;} \\
\ln A_\th  &\text{if  $x \in U(\de_2)$,} 
\end{cases}
$$
and such that
$$
\left|r\frac{df}{dr}\right|
= 
\left|\frac{df}{d(\ln r)}\right|
\leq 1,
$$ 
and 
$$
\left\|r\frac{d}{dr}\left(r\frac{df}{dr}\right)\right\|_{L^\infty}
= 
\left\|\frac{d^2f}{d^2(\ln r)}\right\|_{L^\infty}
\to 0
$$
as $\th\to 0$. It follows that $\lim_{\th\to 0} A_\th=\infty$.
 
After these choices we set $\ep \definedas e^{-A_\th} \de_1$. We
assume that $N$ is obtained from $M$ by a connected sum along $W$ with
parameter $\ep$, as explained in Section \ref{joining}. In particular,
recall that $U^N_\ep(s) = U(s)\setminus U(\ep)/{\sim}$ for all 
$s \geq \ep$. On the set 
$U^N_\ep(\Rmax) = U(\Rmax) \setminus U(\ep)/{\sim}$
we define the variable $t$ by
$$
t \definedas -\ln r_1 + \ln \ep \leq 0
$$
on $U_1(\Rmax)\setminus U(\ep)$ and 
$$
t \definedas \ln r_2 - \ln \ep\geq 0
$$  
on $U_2(\Rmax)\setminus U(\ep)$. This implies
$$
r_i=e^{|t|+ \ln \ep}= \ep e^{|t|}.
$$
The choices imply that $t:U^N_\ep(\Rmax) \to \mR$ is a smooth
function with $t\leq 0$ on $U^N_\ep(\Rmax)\cap M_1$, $t\geq 0$ 
on $U^N_\ep(\Rmax)\cap M_2$, and $t=0$ is the common boundary 
$\pa U_1(\ep)$ identified in $N$ with $\pa U_2(\ep)$.
Then Equation \eref{metric=product} tells us that
$$
r^{-2} g = \ep^{-2} e^{-2|t|}h_i + dt^2 + \sigma^{n-k-1}.
$$

Expressed in the new variable $t$ we have
$$
F(x) = \ep^{-1}e^{-|t|}
$$ 
if $x \in U_\ep^N(R_0) \setminus U_\ep^N(\th)$ or in other words if 
$|t|+\ln \ep\leq \ln R_0$,
and 
$$
f(t)  =  
\begin{cases}
-|t|-\ln\ep     &\text{if $|t|+\ln \ep \in (\th,\Rmax)$,} \\
\ln A_\th  &\text{if  $|t|+\ln \ep \leq \ln \de_2$,}
\end{cases}
$$
and $|df/dt|\leq 1$, $\|d^2f/dt^2\|_{L^\infty}\to 0$. 
After choosing a cut-off function $\chi:\mR\to [0,1]$ such that
$\chi=0$ on $(-\infty,-1]$ and $\chi=1$ on $[1,\infty)$, we define
$$
g_{\th}(x) 
\definedas  
\begin{cases}
F^2 g_i 
&\text{if $x \in M_i \setminus U_i(\th)$;} \\
e^{2f(t)}h_i + dt^2 + \sigma^{n-k-1} 
&\text{if $x \in  U_i(\th)\setminus U_i(\de_1)$;} \\
A_{\th}^2 \chi( A_{\th}^{-1} t ) h_2 
+ A_{\th}^2 (1-\chi(A_{\th}^{-1} t) ) h_1 + dt^2 + \sigma^{n-k-1} 
&\text{if $x \in U_i(\de_1)\setminus U_i(\ep)$.} 
\end{cases}
$$
(Recall that the $h_i$ are defined as the pullback via $w_i$ of 
the metric $g_i$ on $M_i$, composed with restriction to 
$W = W\times \{0\}$.)

%%%%%%%%%%%%%%%%%%%%%%%%%%%%%%%%%%%%%%%%%%%%%%%%%%
\begin{figure} 
\begin{center}
$\framebox{\vbox{
{\sc\large Hierarchy of Variables} 
$$
\Rmax> R_1> R_0 > \th> \de_2>\de_1> \ep>0 
$$
We choose in the order  
$\Rmax,R_1,R_0,\th,\de_2,\de_1, A_{\th}$ 
We can assume for example that 
$\ep=e^{-A_{\th}}\de_1$.
This implies $|t|=A_{\th}\Leftrightarrow r_i=\de_1$.
}}$

\caption{Hierarchy of variables}
\end{center}
\end{figure}
%%%%%%%%%%%%%%%%%%%%%%%%%%%%%%%%%%%%%%%%%%%%%%%%%%

On $U_\ep^N(R_0)$ we write $g_{\th}$ as 
$$
g_\th= \al_t^2  \tilde{h}_t + dt^2 + \sigma^{n-k-1},
$$
where the metric $\tilde{h}_t$ is defined for $t \in \mR$ by 
\begin{equation} \label{def.tildeh}
\tilde{h}_t 
\definedas 
\chi(A_{\th}^{-1} t ) h_2 + (1-\chi(A_{\th}^{-1} t )) h_1,
\end{equation}
and where 
\begin{equation} \label{def.al_t}
\al_t \definedas e^{f(t)}.
\end{equation} 

The rest of the proof consists of showing that $(g_\th)$ is
the desired family of metrics. 
We first choose a sequence $(\th_i)_{i\in \mN}$ converging to $0$ so that
$\lim_{i \to
  \infty}  \la_{\th_i}$ exists.  
 To avoid
complicated notation we write $\theta\to 0$ for the sequence
$(\th_i)_{i\in \mN}$ converging to zero and we will pass successively
to subsequences without changing notation. Similarly $\lim_{\th\to 0}
h(\th)$ should be read as $\lim_{i\to \infty} h(\th_i)$.
We set $\la \definedas \lamin(M_1\amalg M_2,g)$, 
$\la_\th\definedas \lamin(N,g_\th)$, and 
$\bar{\la} \definedas \lim_{\th\to 0} \la_{\th}$. 
Let $J \definedas J^g$ and $J_\th\definedas J^{g_{\th}}$
be the functionals associated respectively 
to $g$ and $g_{\th}$.  

The easier part of the argument is to show that 
\begin{equation} \label{easypart}
\bar{\la} \leq \la.
\end{equation}  
For this let $\al>0$ be a small number. We choose a smooth cut-off 
function $\chi_\al: M_1 \amalg M_2 \to [0,1]$ such that 
$\chi_\al = 1$ on $M_1 \amalg M_2 \setminus U(2 \al)$, 
$|d\chi_\al| \leq 2/\al$, and $\chi_\al = 0$ on $U(\al)$. 
Let $\psi$ be a smooth non-zero spinor such that 
$J(\psi) \leq \la + \de$ where $\de$ is a small positive number. 
On the support of $\chi_\al$ the metrics $g$ and $g_\al$ are 
conformal since $g_\th= F^2 g$ and hence by Formula 
(\ref{funct_conf}) we have
$$
\la_\th 
\leq J_\th\left(\chi_\al \beta^g_{g_\th}(F^{-\frac{n-1}{2}} \psi )
\right)  
= J(\chi_\al \psi)
$$
for $\th< \al$. Proceeding exactly as in the first part of the proof
of Lemma \ref{approx} we show that 
$\lim_{\al \to 0} J(\chi_\al \psi) = J(\psi) \leq \la + \de$. From
this Relation (\ref{easypart}) follows. 

Now we turn to the more difficult part of the proof, that 
\begin{equation} \label{diffpart}
\bar{\la} \geq \min \{ \la, \La_{n,k} \}.
\end{equation}
By Proposition \ref{aubin} we can assume that 
$\la_\th< \lamin(S^n,\si^n)$ for all $\th$, otherwise Relation 
(\ref{diffpart}) is trivial. From Theorem \ref{attained} we know that 
there exists a spinor field $\psi_\th\in \Gamma(\Si^{g_\al} N)$ of 
class $C^2$ such that 
\begin{equation*} 
\int_N |\psi_\th|^\frac{2n}{n-1} \,dv^{g_\th} = 1
\end{equation*}
and
\begin{equation} \label{eqa}
D^{g_{\th}} \psi_\th= \la_\th|\psi_\th|^{\frac{2}{n-1}} \psi_\th.
\end{equation}
We let $x_\th$ in $N$ be such that $|\psi_\th(x_\th)| = m_\th$ where
$m_\th\definedas \| \psi_\th\|_{L^{\infty}(N)}$. 

The proof continues divided in cases. 

%%%%%%%%%%%%%%%%%%%%%%%%%%%%%%%%%%%%%%%%%%%%%%%%%%%%%%%%%%%%%%
\begin{caseI}
The sequence $(m_\th)$ is not bounded.
\end{caseI}
%%%%%%%%%%%%%%%%%%%%%%%%%%%%%%%%%%%%%%%%%%%%%%%%%%%%%%%%%%%%%%

After taking a subsequence, we can assume that 
$\lim_{\th\to 0} m_\th= \infty$. We consider two subcases.

%%%%%%%%%%%%%%%%%%%%%%%%%%%%%%%%%%%%%%%%%%%%%%%%%%%%%%%%%%%%%%
\begin{subcaseI.1}
There exists $a>0$ such that $x_\th\in N \setminus U^N(a)$ 
for an infinite number of $\th$.
\end{subcaseI.1}
%%%%%%%%%%%%%%%%%%%%%%%%%%%%%%%%%%%%%%%%%%%%%%%%%%%%%%%%%%%%%%

We recall that $N \setminus U^N(a) = N_\ep \setminus U^N_\ep(a) =
M_1 \amalg M_2\setminus U(a)$.
By taking a subsequence we can assume that there exists
$\bar{x} \in M_1 \amalg M_2 \setminus U(a)$ such that 
$\lim_{\th\to 0} x_\th= \bar{x}$. We let 
$g_\th' \definedas m_\th^{\frac{4}{n-1}} g_\th$.
In a neighbourhood $U$ of $\bar{x}$ the metric $g_{\th}= F^2 g$ does
not depend on $\th$. We apply Lemma \ref{diffeom} with $O=U$, $\al = \th$, 
$p_\al = x_\th$, $p = \bar{x}$, $\ga_\al = g_\th=F^2 g$, and 
$b_\al = m_\th^{\frac{2}{n-1}}$. Let $r>0$. For $\th$ small enough 
Lemma \ref{diffeom} gives us diffeomorphisms 
$$
\Th_\th: 
B^n(r)
\to
B^{g_\th} ( x_\th, m_\th^{-\frac{2}{n-1}} r) 
$$ 
such that the sequence of metrics $(\Th_\th^* (g_\th'))$ tends to 
the Euclidean metric $\xi^n$ in $C^1(B^n(r))$. We let 
$\psi_\th' \definedas m_\th^{-1} \psi_\th$. By (\ref{confD}) we 
then have
$$
D^{g_\th'} \psi_\th' = \la_\th|\psi_\th'|^\frac{2}{n-1} \psi_\th'
$$
on $B^{g_\th} ( x_\th, m_\th^{-\frac{2}{n-1}} r)$ and 
\begin{equation*}
\begin{split}
\int_ {B^{g_\th} ( x_\th, m_{\th}^{-\frac{2}{n-1}} r)}
|\psi_\th'|^\frac{2n}{n-1}\,dv^{g_\th'} 
&=  
\int_{B^{g_\th} ( x_\th, m_\th^{-\frac{2}{n-1}} r)}
|\psi_\th|^\frac{2n}{n-1}\,dv^{g_\th} \\
&\leq 
\int_N |\psi_\th|^\frac{2n}{n-1} dv^{g_\th} \\ 
&= 1.
\end{split}
\end{equation*}
Here we used the fact that 
$dv^{g_\th'} = m_\th^\frac{2n}{n-1}\,dv^{g_\th}$. 
Since 
$$
\Th_\th: 
(B^n(r), \Th_\th^* (g_\th'))
\to
(B^{g_\th} ( x_\th, m_\th^{-\frac{2}{n-1}} r), g_\th') 
$$ 
is an isometry we can consider $\psi_\th'$ as a solution of 
$$
D^{\Th_\th^*(g_\th')} \psi_\th' 
= 
\la_\th|\psi_\th'|^\frac{2}{n-1} \psi_\th'
$$ 
on $B^n(r)$ with 
$\int_{B^n(r)} |\psi_\th'|^\frac{2n}{n-1} 
\, dv^{\Th_\th^*(g_\th')} \leq 1$. 
Since 
$\| \psi_\th\|_{L^\infty(B^n(r))} = |\psi_\th'(0)| = 1$ 
we can apply Lemma \ref{lim_sol} with  $V = \mR^n$, $\al=\th$, 
$g_\al = \Th_\th^*(g_\th')$, and $\psi_\al = \psi_\th'$ 
(we may apply this lemma since each compact set of $\mR^n$ 
is contained in some ball $B^n(r)$). This shows that there
exists a spinor $\psi$ of class $C^1$ on $(\mR^n,\xi^n)$ which 
satisfies 
$$
D^{\xi^n} \psi = \bar{\la} |\psi|^{\frac{2}{n-1}} \psi.
$$
Furthermore by (\ref{normlr_lim}) we have 
$$
\int_{ B^n(r)} |\psi|^\frac{2n}{n-1} \, dv^{\xi^n} 
= 
\lim_{\th\to 0} \int_{ B^{g_\th} 
( x_\th, m_\th^{-\frac{2}{n-1}} r)} |\psi_\th|^\frac{2n}{n-1}
\,dv^{g_\th}
\leq 1
$$ 
for any $r>0$. We conclude that
$\int_{\mR^n} |\psi|^\frac{2n}{n-1} \, dv^{\xi^n}  \leq 1$. 
Since $|\psi(0)|=1$ we also see that $\psi$ is not identically zero. 
As $(\mR^n, \xi^n)$ and $(S^n \setminus \{ \rm{pt} \}, \sigma^n)$
are conformal we can write $\sigma^n = \Phi^2 \xi^n$ for a positive
function $\Phi$. We define 
$\phi \definedas \Phi^{-\frac{n-1}{2}} \beta^{\xi^n}_{\si^n} \psi$. 
By Equation (\ref{confD}) it follows that 
$\phi \in L^{\frac{2n}{n-1}}(S^n)$ is a solution of
\begin{equation} \label{eqsn} 
D^{\sigma^n} \phi = \bar{\la} |\phi|^{\frac{2}{n-1}} \phi
\end{equation}
on $S^n \setminus \{ \rm{pt} \}$ of class $C^1$.
By Corollary \ref{removal} we know that $\phi$ can be extended to a
weak solution of (\ref{eqsn}) on all $S^n$ and by standard regularity
theorems it follows that $\phi \in C^1(S^n)$. Let $J^{\si^n}$ be the
functional associated to $(S^n,\sigma^n)$. By Equation (\ref{eqsn}) we
have
$$
\lamin(S^n, \sigma^n) \leq J^{\si^n}(\phi) = \bar{\la}
$$
where the inequality comes from Proposition \ref{aubin}.
We have proved Relation (\ref{diffpart}) in this subcase.

%%%%%%%%%%%%%%%%%%%%%%%%%%%%%%%%%%%%%%%%%%%%%%%%%%%%%%%%%%%%%%
\begin{subcaseI.2}
For all $a>0$ it holds that 
$x_\th\notin M_1 \amalg M_2 \setminus U(a)$ for $\th$ 
sufficiently small.
\end{subcaseI.2}
%%%%%%%%%%%%%%%%%%%%%%%%%%%%%%%%%%%%%%%%%%%%%%%%%%%%%%%%%%%%%%

This means that $x_\th$ belongs to $U^N(a)$ if $\th$ is sufficiently
small. This subset is diffeomorphic to $W \times I \times S^{n-k-1}$ 
where $I$ is an interval. Through this diffeomorphism $x_\th$ can be 
written as  
$$
x_\th= (y_\th, t_\th, z_\th)
$$
where $y_\th\in W$, 
$t_\th\in (-\ln R_0  + \ln \ep, -\ln \ep + \ln R_0 )$, and 
$z_\th\in S^{n-k-1}$. By taking a subsequence we can assume that 
$y_\th$, $\frac{t_\th}{A_{\th}}$, and $z_\th$ converge respectively to 
$y \in W$, $T \in [-\infty, +\infty]$, and $z \in S^{n-k-1}$.
We apply Lemma \ref{diffeom} with 
$V = W$, $\al=\th$, $p_\al= y_\th$, $p=y$, $\ga_\al= \tilde{h}_{t_\th}$, 
$\ga_0 = \tilde{h}_T$ (we define $\tilde{h}_{-\infty} \definedas h_1$ 
and $\tilde{h}_{+\infty} \definedas h_2$), and 
$b_\al = m_\th^{\frac{2}{n-1}} \al_{t_\th}$.  The lemma provides
diffeomorphisms
$$
\Th_\th^y : 
B^k(r)
\to
B^{ \tilde{h}_{t_\th}}
(y_\th,  m_\th^{- {\frac{2}{n-1}}}\al_{t_\th}^{-1} r) 
$$ 
for $r>0$ such that $(\Th_\th^y)^* (m_\th^{\frac{4}{n-1}}
\al_{t_\th}^2 \tilde{h}_{t_\th})$ tends to the Euclidean metric $\xi^k$
on $B^k(r)$ as $\th\to 0$. Next we apply Lemma \ref{diffeom}
with $V= S^{n-k-1}$, $\al = \th$, $p_\al = z_\th$, 
$\ga_\al = \ga_0 = \sigma^{n-k-1}$, and 
$b_\al = m_\th^\frac{2}{n-1}$. For $r'>0$ we get the existence of
diffeomorphisms 
$$
\Th_\th^z : 
B^{n-k-1}(r')
\to
B^{\sigma^{n-k-1}}(z_\th, m_\th^{-\frac{2}{n-1}} r')
$$
such that 
$(\Th_\th^z)^* (m_\th^{\frac{4}{n-1}} \sigma^{n-k-1})$ converges 
to $\xi^{n-k-1}$ on $B^{n-k-1}(r')$ as $\th\to 0$. For $r,r',r''>0$
we define
\begin{equation*}
\begin{split}
U_\th(r,r',r'') 
&\definedas  
B^{\tilde{h}_{t_\th}}(y_\th,  m_\th^{-\frac{2}{n-1}} \al_{t_\th}^{-1} r)
\times 
[t_\th- m_\th^{-\frac{2}{n-1}} r'', 
t_\th+ m_\th^{-\frac{2}{n-1}} r''] \\  
&\qquad 
\times  B^{\sigma^{n-k-1}}(z_\th, m_\th^{-\frac{2}{n-1}} r')
\end{split}
\end{equation*}
and 
\begin{equation*}
\begin{aligned}
\Th_\th: 
B^k(r) \times [-r'',r''] \times B^{n-k-1}(r') 
&\to 
U_\th(r,r',r'') 
\\
(y,s,z) 
&\mapsto  \left( \Th_\th^y (y), t(s), \Th_\th^z (z) \right),
\end{aligned}
\end{equation*}
where $t(s) \definedas t_\th+ m_\th^{\frac{2}{n-1}} s$. By
construction $\Th_\th$ is a diffeomorphism. As is readily seen
\begin{equation} \label{diffe}
\Th_\th^* (m_\th^\frac{4}{n-1} g_\th) 
= 
(\Th_\th^y)^* (m_\th^\frac{4}{n-1} \al_{t}^2  \tilde{h}_{t}) 
+ ds^2 +  
(\Th_\th^z)^*(m_\th^\frac{4}{n-1} \sigma^{n-k-1}).
\end{equation}
By construction of $\al_t$ one can verify that 
$$
\lim_{\th\to 0} 
\left\| 
\frac{\al_{t_\th}}{\al_t} - 1 
\right\|_{C^1( [t_\th- m_\th^{-\frac{2}{n-1}} r'', 
t_\th+ m_\th^{-\frac{2}{n-1}} r''])} 
=0
$$
for all $R>0$ since $\frac{df}{dt}$ and $\frac{d^2f }{dt^2}$ are  
uniformly bounded. Moreover it is clear that 
$$ 
\lim_{\th\to 0} 
\left| 
\tilde{h}_{t} - \tilde{h}_{t_\th} 
\right|_{C^1( B^{\tilde{h}_{t_\th}}
(y_\th, m_\th^{-\frac{2}{n-1}} \al_{t_\th}^{-1} R))} 
= 0 
$$
uniformly in $t \in [t_\th- m_\th^{-\frac{2}{n-1}} r'', 
t_\th+ m_\th^{-\frac{2}{n-1}} r'']$. As a consequence
$$ 
\lim_{\th\to  0} 
\left|
(\Th_\th^y)^* \left(m_\th^{\frac{4}{n-1}}\left(\al_{t}^2
\tilde{h}_{t} - \al_{t_\th}^2 \tilde{h}_{t_\th}\right) \right) 
\right|_{C^1(B^k(r))}
=0
$$ 
uniformly in $t$. This implies that the sequence  
$(\Th_\th^y)^* (m_\th^{\frac{4}{n-1}}\al_{t}^2 \tilde{h}_{t})$
tends to the Euclidean metric $\xi^k$ in $C^1(B^k(r))$ uniformly in
$t$ as $\th\to 0$. From \eref{diffe} we know that the sequence 
$(\Th_\th^z)^* (m_\th^{\frac{4}{n-1}} \sigma^{n-k-1})$ tends 
to the Euclidean metric $\xi^{n-k-1}$ on $B^{n-k-1}(r')$ 
as $\th\to 0$. Returning to (\ref{diffe}) we obtain that the sequence
$\Th_\th^* (m_\th^{\frac{4}{n-1}} g_\th)$ tends to 
$\xi^n = \xi^k + ds^2 + \xi^{n-k-1}$ on 
$B^k(r) \times [-r'',r''] \times B^{n-k-1}(r')$. 
As in Subcase~I.1 we apply Lemma \ref{lim_sol} to get a spinor $\psi$
of class $C^1$ on $\mR^n$ which satisfies
$$
D^{\xi^n} \psi = \bar{\la} |\psi|^{\frac{2}{n-1}} \psi
$$
with $\int_{ B^n (r)} |\psi|^\frac{2n}{n-1} dx \leq 1$ for all 
$r\in \mR^+$. Lemma \ref{lim_sol} tells us that $|\psi(0)|=1$ so 
$\psi$ does not vanish identically.
As in Subcase~I.1 we conclude that
$$
\la \leq \lamin(S^n, \sigma^n) \leq \bar{\la}.
$$
This ends the proof of Theorem \ref{main} in Case I.

%%%%%%%%%%%%%%%%%%%%%%%%%%%%%%%%%%%%%%%%%%%%%%%%%%%%%%%%%%%%%%
\begin{caseII}
There exists a constant $C_1$ such that $m_\th\leq C_1$ 
for all $\th$. 
\end{caseII}
%%%%%%%%%%%%%%%%%%%%%%%%%%%%%%%%%%%%%%%%%%%%%%%%%%%%%%%%%%%%%%

Again we consider two subcases.

%%%%%%%%%%%%%%%%%%%%%%%%%%%%%%%%%%%%%%%%%%%%%%%%%%%%%%%%%%%%%%
\begin{subcaseII.1}
Assume that  
\begin{equation} \label{assomp}
\liminf_{\th\to 0} \int_{N \setminus U^N(a)} 
|\psi_\th|^{\frac{2n}{n-1}}\,dv^{g_\th} >0
\end{equation}
for some number $a>0$.
\end{subcaseII.1}
%%%%%%%%%%%%%%%%%%%%%%%%%%%%%%%%%%%%%%%%%%%%%%%%%%%%%%%%%%%%%%

Let $K$ a compact subset such that 
$K \subset M_1 \amalg M_2 \setminus W'$. 
Choose a small number $b$ such that 
$K \subset M_1 \amalg M_2 \setminus U(2b)= N \setminus U^N(2b)$. 
Let $\chi \in C^{\infty}(M_1 \amalg M_2)$, $0 \leq \chi \leq 1$, be
a cut-off function equal to $1$ on $M_1 \amalg M_2 \setminus U(2b)$
and equal to $0$ on $U(b)$. Set $\psi_\th' \definedas 
F^{\frac{n-1}{2}} (\beta^g_{g_\th})^{-1} \psi_\th$.
Since $g_{\th}= F^2 g$ on the support of $\chi$ we have  
\begin{equation*} 
D^g \psi_\th' = \la_\th|\psi_\th'|^\frac{2}{n-1} \psi_\th'
\end{equation*}
on this set. For $r>0$ we have
\begin{equation*}
\begin{split}
&\int_{M_1 \amalg M_2} |D^g (\chi \psi_\th')|^r\,dv^g \\
&\qquad
= 
\int_{M_1 \amalg M_2} \left| 
\grad^g \chi \cdot \psi_\th' 
+ \chi \la_\th|\psi_\th'|^\frac{2}{n-1} \psi_\th' 
\right|^r \,dv^g \\
&\qquad
\leq  
2^r \left( 
\int_{M_1 \amalg M_2} | \grad^g \chi|^r |\psi_\th'|^r \,dv^g 
+ \la_\th^r \int_{M_1 \amalg M_2} \chi^r
|\psi_\th'|^{\frac{(n+1)r}{n-1}}\,dv^g 
\right) \\
&\qquad
\leq C.
\end{split}
\end{equation*}
since $m_\th\leq C_1$. Together with Relation (\ref{ellip_est2}) we 
get that the sequence $(\chi \psi_\th')$ is bounded in 
$H_1^r({M_1 \amalg M_2})$ for all $r>0$. Proceeding as in the proof 
of Lemma \ref{lim_sol} we get a $C^1$ spinor $\psi_0$ defined on $K$ 
such that a subsequence of $(\psi_\th')$ converges to $\psi_0$ in 
$C^0(K)$ and which satisfies 
\begin{equation} \label{eqlim_case2}
D^g \psi_0 = \bar{\la} |\psi_0|^\frac{2}{n-1} \psi_0.
\end{equation}
Furthermore the convergence in $C^0$ implies that 
$$
\int_K |\psi_0|^\frac{2n}{n-1}dv^g 
\leq 
\liminf_{\th\to 0} 
\int_K |\psi_\th'|^\frac{2n}{n-1}dv^g 
= 
\liminf_{\th\to 0}
\int_K |\psi_\th|^\frac{2n}{n-1}dv^{g_\th}
\leq 1.
$$
Repeating the same for a sequence of compact sets which exhausts
${M_1 \amalg M_2} \setminus W'$ and taking a diagonal 
subsequence we can extend $\psi_0$ to ${M_1 \amalg M_2} \setminus
W'$. Since $\psi_0 \in L^\frac{2n}{n-1}({M_1 \amalg M_2} \setminus W')
= L^\frac{2n}{n-1}({M_1 \amalg M_2})$ we can use Theorem \ref{removal}
to extend $\psi_0$ to a weak solution of Equation (\ref{eqlim_case2})
on $M_1 \amalg M_2$. Note here that since $D^g$ is invertible we have 
$\bar{\la} > 0$. By standard regularity theorems we conclude that 
$\psi_0 \in C^1({M_1 \amalg M_2})$. By (\ref{assomp}) we have  
\begin{equation*}
\begin{split}
\int_{ {M_1 \amalg M_2} \setminus U(a)}
|\psi_0|^{\frac{2n}{n-1}}\,dv^g  
&= \lim_{\th\to 0} \int_{{M_1 \amalg M_2} \setminus U(a)} 
|\psi_\th'|^{\frac{2n}{n-1}}\,dv^g \\
&= \lim_{\th\to 0} \int_{{M_1 \amalg M_2} \setminus U(a)} 
|\psi_\th|^{\frac{2n}{n-1}}\,dv^{g_\th} \\
&> 0,
\end{split}
\end{equation*}
and we conclude that $\psi_0$ does not vanish identically. 
Equation (\ref{eqlim_case2}) then leads to
$$
\la \leq J(\psi_0) 
= 
\bar{\la} \left(\int_{M_1 \amalg M_2} |\psi_0|^{\frac{2n}{n-1}}\,dv^g
\right)^{\frac{n+1}{n} -1} 
\leq
\bar{\la},
$$
which proves Theorem \ref{main} in this case.

%%%%%%%%%%%%%%%%%%%%%%%%%%%%%%%%%%%%%%%%%%%%%%%%%%%%%%%%%%%%%%
\begin{subcaseII.2}
We have 
\begin{equation} \label{assomp2} 
\liminf_{\th\to 0} \int_{N \setminus U^N(a)} 
|\psi_\th|^{\frac{2n}{n-1}}\,dv^{g_\th}= 0
\end{equation}
for all $a>0$.
\end{subcaseII.2}
%%%%%%%%%%%%%%%%%%%%%%%%%%%%%%%%%%%%%%%%%%%%%%%%%%%%%%%%%%%%%%

This case is the most difficult one and we proceed in several steps.
The assumption here is that we have a sequence $(\th_i)$ which tends
to zero as $i \to\infty$ with the property that the integral above
tends to zero for all $a>0$. We will abuse notation and write 
$\lim_{\th\to 0}$ for what should be a limit as $i \to \infty$ or a 
limit of a subsequence.

For positive $a$ and $\th$ let 
$$
\ga_\th(a) 
\definedas
\frac{\int_{N \setminus U^N(a)} |\psi_\th|^2 \,dv^{g_\th} }
{\int_{U^N(a)} |\psi_\th|^2 \,dv^{g_\th}}
$$
The first step is to establish an estimate for $\ga_\th(a)$.

%%%%%%%%%%%%%%%%%%%%%%%%%%%%%%%%%%%%%%%%%%%%%%%%%%%%%%%%%%%%%%
%Step 1
%
\begin{step}
There is a constant $C_0$ so that 
\begin{equation} \label{l2_linf}
1 \leq C_0 \left(\ga_\th(a) 
+ \| \psi_\th\|^{\frac{4}{n-1}}_{L^\infty(U^N(2a))}  \right)
\end{equation}
for all $a>0$.
\end{step}
%%%%%%%%%%%%%%%%%%%%%%%%%%%%%%%%%%%%%%%%%%%%%%%%%%%%%%%%%%%%%%

Let $\chi \in C^\infty(N)$, $0\leq \chi \leq 1$, be a cut-off
function with $\chi = 1$ on $U^N(a)$ and $\chi = 0$ on 
$N \setminus U^N(2a) = {M_1 \amalg M_2} \setminus U(2a)$. Since the
definitions of $U^N(a)$ and $U(a)$ use the distance to $W'$ for the
metric $g$ we can and do assume that $|d \chi|_g \leq 2/a$. For
the metric $g_\th$ this gives 
$$
|d \chi|_{g_\th} 
= F^{-1} |d \chi|_g
= r |d \chi|_{g} 
\leq 2a \frac{2}{a} 
= 4.
$$
From Lemma \ref{sphere_bundle} and Equation (\ref{eqa}) it
follows that   
\begin{equation*}
\begin{split}
\frac{(n-k-1)^2 }{4}
&\leq 
\frac{\int_N |D^{g_\th} (\chi \psi_\th) |^2 \,dv^{g_\th}}
{\int_N| \chi \psi_\th|^2 \, dv^{g_\th} } \\
&= 
\frac{\int_N |d\chi|^2_{g_\th} |\psi_\th|^2 \,dv^{g_\th} 
+ \la_\th^2 \int_N \chi^2 
|\psi_\th|^{\frac{2(n+1)}{n-1}} \,dv^{g_\th}}
{\int_N| \chi \psi_\th|^2 \,dv^{g_\th} } \\
&\leq  
\frac{ 16 \int_{U^N(2a) \setminus U^N(a)}|\psi_\th|^2 \,dv^{g_\th} + 
\la_\th^2 \| \psi_\th\|^{\frac{4}{n-1}}_{L^\infty(U^N(2a))} 
\int_{N} |\chi \psi_\th|^2 \,dv^{g_\th}} 
{\int_{N} |\chi\psi_\th|^2 \,dv^{g_\th}} \\
&\leq  
\frac{16\int_{U^N(2a) \setminus U^N(a)}|\psi_\th|^2 \,dv^{g_\th}}
{\int_{U^N(a)} |\psi_\th|^2\,dv^{g_\th}} 
+ \la_\th^2 \| \psi_\th\|^{\frac{4}{n-1}}_{L^\infty(U^N(2a))} \\
&\leq   
16 \ga_\th(a) 
+ \la_\th^2  \| \psi_\th\|^{\frac{4}{n-1}}_{L^\infty(U^N(2a))}.
\end{split} 
\end{equation*}
Using that $\la_\th\leq \lamin(S^n,\si^n)$ by Proposition \ref{aubin} 
%and 
%$$ 
%\| \psi_\th\|_{L^\infty(U^N(2a))} \leq m_\th\leq C_1,
%$$
we obtain Relation (\ref{l2_linf}) with
$$
C_0 
\definedas
\frac{4}{(n-k-1)^2} \max 
\left\{ 
16, \lamin(S^n)^2
\right\}.  
$$
This ends the proof of Step~1. 

%%%%%%%%%%%%%%%%%%%%%%%%%%%%%%%%%%%%%%%%%%%%%%%%%%%%%%%%%%%%%%
%Step 2
%
\begin{step} 
There exist a sequence of positive numbers $(a_\th)$ which
tends to $0$ with $\th$ and constants $0 < m < M$ such that  
\begin{equation} \label{norminfbound}
m 
\leq  
{\| \psi_\th\|}_{L^\infty(U^N(2a_\th))} 
\leq 
M
\end{equation}
for all $\th$.
\end{step}
%%%%%%%%%%%%%%%%%%%%%%%%%%%%%%%%%%%%%%%%%%%%%%%%%%%%%%%%%%%%%%

By (\ref{assomp2}) we have
$$
\lim_{\th\to 0} \int_{N \setminus U^N(a)} 
|\psi_\th|^{\frac{2n}{n-1}} \,dv^{g_\th}
= 0
$$
for all $a > 0$. Since $\vol\left(N \setminus U^N(a),g_\th\right)$
does not depend on $\th$ if $\th< a$ it follows that  
$$
\lim_{\th\to 0} {\left( \int_{N \setminus U^N(a)}
|\psi_\th|^{\frac{2n}{n-1}}\,dv^{g_\th} \right)}^{\frac{n-1}{n}} 
\vol (N \setminus U^N(a),g_\th)^{\frac{1}{n}} = 0
$$
for all $a$. Hence we can take a sequence $(a_\th)$ which tends
sufficiently slowly to $0$ so that
\begin{equation} \label{choiceofa} 
\lim_{\th\to 0} 
\left( 
\int_{N \setminus U^N(a_\th)} 
|\psi_\th|^{\frac{2n}{n-1}} \,dv^{g_\th} 
\right)^{\frac{n-1}{n}} 
\vol(N \setminus U^N(a_\th),g_\th)^{\frac{1}{n}} 
= 0.
\end{equation}
Using the H\"older inequality we get  
\begin{equation*}
\begin{split}
\ga_\th(a_\th) 
&= 
\frac{\int_{N \setminus U^N(a_\th)} |\psi_\th|^2\,dv^{g_\th}}
{\int_{U^N(a_\th)} |\psi_\th|^2\,dv^{g_\th}} \\
&\leq  
\frac{(\int_{N \setminus U^N(a_\th)} 
|\psi_\th|^{\frac{2n}{n-1}} \,dv^{g_\th} )^{\frac{n-1}{n}} 
\vol (N \setminus U^N(a_\th), g_\th)^{\frac{1}{n}}}
{{\| \psi_\th\|}^{-\frac{2}{n-1}}_{L^\infty(U^N(a_\th))} 
\int_{U^N(a_\th)} |\psi_\th|^{\frac{2n}{n-1}} \,dv^{g_\th}}.
\end{split}
\end{equation*}
The numerator of this expression tends to $0$ by Relation
(\ref{choiceofa}). Further by (\ref{choiceofa}) we have 
\begin{equation*}
\begin{split}
\lim_{\th\to 0} 
\int_{U^N(a_\th)} |\psi_\th|^{\frac{2n}{n-1}} \,dv^{g_\th} 
&=  
\lim_{\th\to 0} \int_N |\psi_\th|^{\frac{2n}{n-1}} \,dv^{g_\th} 
-  
\int_{N \setminus U^N(a_\th)} |\psi_\th|^{\frac{2n}{n-1}} 
\,dv^{g_\th} \\
&= 1.
\end{split}
\end{equation*}
Together with the fact that 
$\| \psi_\th\|_{L^\infty(U^N(a_\th))} \leq m_\th\leq C_1$ 
we obtain that
$$
\lim_{\th\to 0} \ga_\th(a_\th) = 0.
$$
From Relation (\ref{l2_linf}) applied with $a = a_\th$ we know that
$\| \psi_\th\|_{L^\infty(U^N(2a_\th))}$ is bounded from below. 
Moreover, by the assumption of Case II we have that 
$\| \psi_\th\|_{L^\infty(U^N(2a_\th))} \leq m_\th\leq C_1$.
This finishes the proof of Step~2.

%%%%%%%%%%%%%%%%%%%%%%%%%%%%%%%%%%%%%%%%%%%%%%%%%%%%%%%%%%%%%%
%Step 3 -- final step
%
\begin{step} We have 
\begin{equation*} 
\bar{\la} \geq \La_{n,k}.
\end{equation*}
\end{step}
%%%%%%%%%%%%%%%%%%%%%%%%%%%%%%%%%%%%%%%%%%%%%%%%%%%%%%%%%%%%%%

Let $x_\th$ be a point in the closure of $U^N(2a_\th)$ such that 
$|\psi_\th(x_\th)| = \| \psi_\th\|_{L^\infty(U^N(2a_\th))}$. 
As in Subcase~I.2 we write $x_\th= (y_\th, t_\th, z_\th)$ 
where $y_\th\in W$, 
$t_\th\in (-\ln R_0 + \ln \ep, -\ln \ep + \ln R_0)$, and 
$z_\th\in S^{n-k-1}$.
By restricting to a subsequence we can assume that $y_\th$, 
$t_\th/ A_{\th}$, and $z_\th$ converge respectively
to $y \in W$, $T \in [-\infty, +\infty]$, and $z \in S^{n-k-1}$.
We apply Lemma \ref{diffeom} with $V = W$, $\al = \th$, 
$p_\al = y_\th$, $p=y$, $\ga_\al= \tilde{h}_{t_\th}$, 
$\ga_0 = \tilde{h}_T$, and $b_\al = \al_{t_\th}$
(recall that $\tilde{h}_t$ and $\al_t$ were defined in 
\eref{def.tildeh} and \eref{def.al_t})
and conclude that there is a diffeomorphism
$$
\Th_\th^y : 
B^k(r)
\to
B^{ \tilde{h}_{t_\th}}(y_\th, \al_{t_\th}^{-1} r) 
$$ 
for $r>0$ such that 
$({\Th_\th^y})^* (\al_{t_\th}^2 \tilde{h}_{t_\th})$  
converges to the Euclidean metric $\xi^k$ on $B^k(r)$. 
For $r,r'>0$ we define
$$
U_\th(r,r') 
\definedas
B^{\tilde{h}_{t_\th}}(y_\th,\al_{t_\th}^{-1} r) 
\times 
[t_\th- r', t_\th+ r'] \times S^{n-k-1}
$$
and 
\begin{equation*} 
\begin{aligned}
\Th_\th: 
B^k(r) \times [-r',r'] \times S^{n-k-1}
&\to  
U_\th(r,r')
\\
(y,s,z) 
&\mapsto
\left( \Th_\th^y (y), t(s), z \right), 
\end{aligned}
\end{equation*}
where $t(s) \definedas t_\th+ s$. By construction $\Th_\th$ is a
diffeomorphism. Since 
$g_\th= \al_t^2 \tilde{h}_t + dt^2+ \sigma^{n-k-1}$ we see that
\begin{equation} \label{metric=}
\Th_\th^* ( g_\th) 
= 
\frac{\al_t^2}{\al_{t_\th}^2}
({\Th_\th^y})^* ( \al_{t_\th}^2  \tilde{h}_{t}) + ds^2 +  
\sigma^{n-k-1}.
\end{equation}
We will now find the limit of $\Th_\th^* ( g_\th) $ in the
$C^1$ topology. We define 
$\ka \definedas \lim_{\theta\to 0} f'(t_\theta)$.

\begin{lemma} \label{lim_metric} 
The sequence of metrics $\Th_\th^* ( g_\th)$ tends to 
$$
G_\ka
=
\eta^{k+1}_\ka + \sigma^{n-k-1}
= 
e^{2\ka s}\xi^k + ds^2 + \sigma^{n-k-1}
$$ 
in $C^1$ on 
$B^k(r) \times [-r',r'] \times S^{n-k-1}$ for fixed $r,r'>0$.
\end{lemma}

\begin{proof}
Recall that $\al_t=e^{f(t)}$. The intermediate value theorem tells
us that
$$
\left|
f(t) - f(t_\th) - f'(t_\th)(t-t_\th)
\right|
\leq 
\frac{r'^2}{2} 
\max_{\xi \in [t_\th-r',t_\th+r']} 
\left| f''(\xi) \right|
$$
for all $t\in [t_\th-r', t_\th+r']$. On the other hand we assume that 
$f''(t)\to 0$ as $\th\to 0$, so 
$$
\left\|
f(t) - f(t_\th) - f'(t_\th)(t-t_\th)
\right\|_{C^0([t_\th-r',t_\th+r'])} 
\to 0
$$
as $\th\to 0$ (and $r'$ fixed). Furthermore
\begin{equation*}
\begin{split}
\left|
\ddt \left( f(t) - f(t_\th) - f'(t_\th)(t-t_\th) \right)
\right|
&=
\left| f'(t) - f'(t_\th) \right| \\
&= 
\left| \int_{t_\th}^t f''(s) \,ds \right| \\
&\leq 
r' \max_{\xi\in [t_\th-r',t_\th+r']} 
\left| f''(\xi) \right| \\
&\to 0
\end{split}
\end{equation*}
as $\th\to 0$. Together with $\ka=\lim_{\theta\to 0}f'(t_\theta)$
we have
$$
\left\| 
f(t)-f(t_\th)-\ka(t-t_\th)
\right\|_{C^1([t_\th-r',t_\th+r'])}
\to 0.
$$
Exponentiation of functions 
% $[t_\th-r',t_\th+r']\to \mR$ 
is a continuous map  
$$
C^1([t_\th-r',t_\th+r'])
\ni
\ti f \mapsto \exp \circ \ti f
\in
C^1([t_\th-r',t_\th+r']).
$$ 
Hence
$$
\left\|  
\frac{\al_t}{\al_{t_\th}} - e^{\ka(t-t_\th)} 
\right\|_{ C^1([t_\th- r',t_\th+ r']) }
= 
\left\| 
e^{f(t)-f(t_\th)} - e^{\ka(t-t_\th)}
\right\|_{ C^1([t_\th- r',t_\th+ r']) }   
\to 0
$$
as $\th\to 0$. We now write  
$\al^2_t \tilde{h}_t 
= \al^2_t (\tilde{h}_t - \tilde{h}_{t_\th}) 
+ \frac{\al^2_t}{\al^2_{t_\th}} \al^2_{t_\th} \tilde{h}_{t_\th}$. 
Using the fact that 
$$ 
\lim_{\th\to 0}  
\left\| 
\tilde{h}_{t} - \tilde{h}_{t_\th} 
\right\|_{C^1( B_{\tilde{h}_{t_\th}}(y_\th, \al_{t_\th}^{-1} R))} 
= 0 
$$ 
uniformly for $t \in [t_\th-r', t_\th-r']$ we get that the
sequence $\frac{\al_t^2}{\al_{t_\th}^2} 
(\Th_\th^y)^* ( \al_{t_\th}^2 \tilde{h}_{t})$ tends to 
$e^{2\ka s}\xi^k$ in $C^1$ on $B^k(r)$. Going back to Relation 
(\ref{metric=}) this proves Lemma \ref{lim_metric}. 
\end{proof}

We continue with the proof of Step 3. As in subcases I.1 and I.2 
we apply Lemma \ref{lim_sol} with 
$(V,g) = (\mR^{k+1} \times S^{n-k-1}, G_\ka)$, 
$\al = \th$, and $g_\al = \Th_\th^* (g_\th)$ (we can apply this 
lemma since any compact subset of $\mR^{k+1} \times S^{n-k-1}$ is
contained in some $B^k(r) \times [-r',r'] \times S^{n-k-1}$). 
We obtain a $C^1$ spinor $\psi$ which is a solution of 
\begin{equation*} 
D^{G_\ka} \psi = \bar{\la} |\psi|^\frac{2}{n-1} \psi
\end{equation*} 
on $(\mR^{k+1} \times S^{n-k-1}, G_\ka)$. 
From (\ref{normlr_lim}) it follows that
$$
\int_{\mR^{k+1} \times S^{n-k-1}} 
|\psi|^{\frac{2n}{n-1}} \,dv^{G_\ka}
\leq 1.
$$ 
From (\ref{norminf_lim}) it follows that 
$\psi \in L^\infty(\mR^{k+1} \times S^{n-k-1})$, and from 
(\ref{norminf_lim}) and (\ref{norminfbound}) it follows that $\psi$ 
does not vanish identically. We want to show that 
$\psi \in L^2(\mR^{k+1} \times S^{n-k-1})$. From (\ref{normlr_lim}) 
we get that  
\begin{equation} \label{psil2}
\begin{split}
\int_{B^k(r) \times [-r',r'] \times S^{n-k-1}}
|\psi|^2\,dv^{G_\ka}  
&= 
\lim_{\th\to 0} 
\int_{U_\th(r,r')} |\psi_\th|^2 dv^{g_\th} \\
&\leq 
\lim_{\th\to 0}
\int_{U^N(a)} |\psi_\th|^2 dv^{g_\th}
\end{split}
\end{equation}
for some fixed number $a>0$ independent of $r$, $r'$ and $\th$. Let
$\chi$ be defined as in Step~1. Using the H\"older inequality,
Lemma \ref{sphere_bundle}, and Equation (\ref{eqa}) we see that   
\begin{equation*}
\begin{split}
&\frac{(n-k-1)^2 }{4} \\ 
&\qquad
\leq 
\frac{\int_N |D^{g_\th} (\chi \psi_\th) |^2 \,dv^{g_\th}}
{\int_N| \chi \psi_\th|^2 \,dv^{g_\th} } \\ 
&\qquad
= 
\frac{\int_N |d\chi|^2_{g_\th} |\psi_\th|^2 \,dv^{g_\th} +
\la_\th^2 \int_N \chi^2 |\psi_\th|^{\frac{2(n+1)}{n-1}} \,dv^{g_\th}}
{\int_N| \chi \psi_\th|^2 \,dv^{g_\th}} \\
&\qquad
\leq 
\frac{16 \int_{U^N(2a) \setminus U^N(a)}|\psi_\th|^2 \,dv^{g_\th} + 
\la_\th^2 \| \psi_\th\|^{\frac{2}{n-1}}_{L^\infty(U^N(2a))}  
\int_{U^N(2a)} |\psi_\th|^\frac{2n}{n-1} \,dv^{g_\th}} 
{\int_{U^N(a)} |\psi_\th|^2\,dv^{g_\th} }.
\end{split}
\end{equation*} 
We have 
$$ 
\la_\th^2 \| \psi_\th\|^{\frac{2}{n-1}}_{L^\infty(U^N(2a))}  
\int_{U^N(2a)} |\psi_\th|^\frac{2n}{n-1} \,dv^{g_\th} 
\leq
\lamin(S^n,\si^n)^2 C_1^\frac{2}{n-1}
$$
and 
\begin{equation*}
\begin{split}
&\int_{U^N(2a) \setminus U^N(a)}|\psi_\th|^2 \,dv^{g_\th} \\
&\qquad
\leq
{\left( 
\int_{U^N(2a) \setminus U^N(a)}|\psi_\th|^\frac{2n}{n-1} \,dv^{g_\th} 
\right)}^{\frac{n-1}{n}} 
\vol \left(  U^N(2a) \setminus  U^N(a), g_\th\right)^{\frac{1}{n}} \\
&\qquad
\leq 
\vol\left(  U^N(2a) \setminus U^N(a), g_\th\right)^{\frac{1}{n}}.
\end{split}
\end{equation*} 
Since $g_\th$ does not depend on $\th$ on $U^N(2a) \setminus U^N(a)$
for $\th<a$, we get the existence of a constant $C$
such that 
$$
\frac{(n-k-1)^2 }{4} 
\leq 
\frac{C}{\int_{U^N(a)} |\psi_\th|^2\,dv^{g_\th}}.
$$
Together with (\ref{psil2}) we obtain that 
$$
\int_{B^k(r) \times [-r',r'] \times S^{n-k-1}} 
|\psi|^2\,dv^{G_\ka} 
\leq C
$$
where $C$ is independent of $r$ and $r'$. This proves that 
$\psi \in L^2 (\mR^{k+1} \times S^{n-k-1})$. Since
the spinor $\psi$ is non-zero and 
$$
\psi \in  L^\infty(\mR^{k+1} \times S^{n-k-1}) 
\cap  \Cloc^1(\mR^{k+1} \times S^{n-k-1}) 
\cap L^2(\mR^{k+1} \times S^{n-k-1})
$$
with 
$$
\int_{\mR^{k+1} \times S^{n-k-1}} 
|\psi|^{\frac{2n}{n-1}} \,dv^{G_\ka}
\leq 1
$$ 
we get that $\bar{\la} \geq \La_{n,k}$ by the definition of
$\La_{n,k}$. This ends the proof of this subcase and the proof of
Theorem \ref{main}.

%%%%%%%%%%%%%%%%%%%%%%%%%%%%%%%%%%%%%%%%%%%%%%%%%%%%%%%%%%%%%%
\subsection{Proof of Theorem \ref{lamin_hyperb}}
%%%%%%%%%%%%%%%%%%%%%%%%%%%%%%%%%%%%%%%%%%%%%%%%%%%%%%%%%%%%%%

We prove Theorem \ref{lamin_hyperb} by contradiction. Assume that
there is a sequence $\ka_i \in [-1,1]$, ${i \in \mN}$, for which 
$$
\lim_{i \to \infty}
\lamintilde(\mR^{k+1} \times S^{n-k-1}, 
G_{\ka_i}) = 0.
$$
After removing the indices $i$ for which $\lamintilde$ is infinite
we have for all $i$ a solution of
\begin{equation} \label{eqpsii}
D^{G_{\ka_i}
} \psi_i =\la_i |\psi_i|^{\frac{2}{n-1}} \psi_i
\end{equation} 
where $\la_i \to 0$ as $i \to \infty$. Moreover, the spinors 
$\psi_i$ are in $L^\infty \cap L^2 \cap \Cloc^1$ and
\begin{equation*} 
\int_{\mR^{k+1} \times  S^{n-k-1}} 
|\psi_i|^{\frac{2n}{n-1}} \,dv^{G_{\ka_i}
} 
\leq 1.
\end{equation*}
Let $m_i  \definedas \| \psi_i \|_{L^\infty}$. We cannot assume 
that $m_i$ is attained, but since 
$(\mR^{k+1} \times S^{n-k-1}, G_{\ka_i})$ 
is a symmetric space we can compose $\psi_i$ with isometries so 
that $|\psi_i(P)|> m_i/2$ for some fixed point 
$P \in \mR^{k+1} \times  S^{n-k-1}$. First we prove that  
\begin{equation} \label{sup=infty}
\lim_{i \to \infty} m_i = \infty.
\end{equation} 
By Lemma \ref{sphere_bundle} and Equation (\ref{eqpsii}) we have 
\begin{equation*} 
\begin{split}
\frac{(n-k-1)^2}{4} 
&\leq 
\frac{\int_{\mR^{k+1} \times S^{n-k-1}}  
|D^{G_{\ka_i}
} \psi_i|^2 \,dv^{G_{\ka_i}
} }
{\int_{\mR^{k+1} \times S^{n-k-1}} 
|\psi_i|^2\,dv^{G_{\ka_i}
} } \\ 
&\leq 
\frac{\la_i^2 \int_{\mR^{k+1} \times S^{n-k-1}}  
|\psi_i|^\frac{2(n+1)}{n-1} \,dv^{G_{\ka_i}
}}
{\int_{\mR^{k+1} \times S^{n-k-1}}  
|\psi_i|^2 \,dv^{G_{\ka_i}
} } \\
&\leq 
\la_i^2 m_i^{\frac{4}{n-1}}.
\end{split}
\end{equation*}
Since $\la_i$ tends to zero this proves (\ref{sup=infty}).
Restricting to subsequence we can assume that 
$\lim_{i \to \infty} \ka_i$ exists and we denote this limit by 
$\ka \in [-1,1]$. We apply Lemma \ref{diffeom} with $\al = 1/i$, 
$(V, \ga_\al) = (\mR^{k+1} \times S^{n-k-1}, 
G_{\ka_i})$, 
$(V, \ga_0) = (\mR^{k+1} \times S^{n-k-1}, G_\ka)$,
$p_\al = p = P$, and $b_\al = m_i^\frac{2}{n-1}$. For $r>0$ we
obtain a diffeomorphism 
$$
\Th_i : 
B^n(r)
\to
B^{G_{\ka_i}
} (P, m_i^\frac{2}{n-1}r) 
$$
such that $\Th_i^* (m_i^\frac{4}{n-1} 
(G_{\ka_i}))$ tends to the
Euclidean metric $\xi^n$ on $B^n(r)$. Proceeding as in
Subcase~I.1 of Theorem \ref{main} we construct a non-zero spinor
$\psi$ belonging to $L^{\frac{2n}{n-1}}(\mR^n)$ such that
$$
D^{\xi^n} \psi 
= 
\lim_{i \to \infty} \la_i |\psi|^\frac{2}{n-1} \psi 
= 
0.
$$
Again as in Subcase~I.1 of Theorem \ref{main} we get 
$0 \geq \lamin(S^n, \si^n)$, which is false. This proves Theorem
\ref{lamin_hyperb}.

%%%%%%%%%%%%%%%%%%%%%%%%%%%%%%%%%%%%%%%%%%%%%%%%%%%%%%%%%%%%%%
%%%%%%%%%%%%%%%%%%%%%%%%%%%%%%%%%%%%%%%%%%%%%%%%%%%%%%%%%%%%%%
%%%%%%%%%%%%%%%%%%%%%%%%%%%%%%%%%%%%%%%%%%%%%%%%%%%%%%%%%%%%%%

%\bibliographystyle{amsplain}
%\bibliography{literatur}

\providecommand{\bysame}{\leavevmode\hbox to3em{\hrulefill}\thinspace}
\providecommand{\MR}{\relax\ifhmode\unskip\space\fi MR }
\providecommand{\MRhref}[2]{%
  \href{http://www.ams.org/mathscinet-getitem?mr=#1}{#2}
}
\providecommand{\href}[2]{#2}

%%%%%%%%%%%%%%%%%%%%%%%%%%%%%%%%%%%%%%%%%%%%%%%%%%%%%%%%%%%%%%%%%%     
\end{document}